\documentclass[12pt]{article}
\usepackage{amsmath,amsthm,mathrsfs}
\usepackage{amsfonts}
\usepackage{latexsym}
\usepackage{cite}

\usepackage{t1enc}
\usepackage[latin1]{inputenc}
\usepackage[english]{babel}
\usepackage[dvips]{graphicx}
\usepackage{graphicx}
\usepackage[natural]{xcolor}
\newtheorem{theorem}{Theorem}

\newtheorem{proposition}[theorem]{Proposition}
\newtheorem{conjecture}[theorem]{Conjecture}

\newtheorem{problem}[theorem]{Problem}

\newcommand{\diam}{{\rm diam}}

\begin{document}

\title{Non-$\ell$-distance-balanced generalized Petersen graphs $GP(n,3)$ and $GP(n,4)$}

\author{Gang Ma$^{a}$ \and Jianfeng Wang$^{a,}$\footnote{Corresponding author} \and Sandi Klav\v{z}ar$^{b,c,d}$}

\date{}

\maketitle
\vspace{-0.8 cm}
\begin{center}
$^a$ School of Mathematics and Statistics, Shandong University of Technology\\
 Zibo, China\\
{\tt math$\_$magang@163.com\\
jfwang@sdut.edu.cn}\\
\medskip

$^b$ Faculty of Mathematics and Physics, University of Ljubljana, Slovenia\\
{\tt sandi.klavzar@fmf.uni-lj.si}\\
\medskip

$^c$ Faculty of Natural Sciences and Mathematics, University of Maribor, Slovenia\\
\medskip

$^d$ Institute of Mathematics, Physics and Mechanics, Ljubljana, Slovenia\\
\end{center}

\begin{abstract}
A connected graph $G$ of diameter ${\rm diam}(G) \ge \ell$  is $\ell$-distance-balanced if $|W_{xy}|=|W_{yx}|$ for every $x,y\in V(G)$ with $d_{G}(x,y)=\ell$, where $W_{xy}$ is the set of vertices of $G$ that are closer to $x$ than to $y$.
We prove that the generalized Petersen graph $GP(n,3)$ where $n>16$ is not $\ell$-distance-balanced for any $1\le \ell<\diam(GP(n,3))$,
and $GP(n,4)$ where $n>24$ is not $\ell$-distance-balanced for any $1\le \ell<\diam(GP(n,4))$.
This partially solves a conjecture posed by \v{S}.~Miklavi\v{c} and P.~\v{S}parl
(Discrete Appl.\ Math.\ 244:143--154, 2018).
\end{abstract}


\medskip\noindent
{\bf AMS Subj.\ Class:} 05C12

\section{Introduction}
\label{S:intro}

If $G = (V(G), E(G))$ is a connected graph and $x, y\in V(G)$, then the {\it distance} $d_{G}(x, y)$ between $x$ and $y$ is the number of edges on a shortest $x,y$-path. The diameter $\diam(G)$ of $G$ is the maximum distance between its vertices. The set $W_{xy}$ contains the vertices that are closer to $x$ than to $y$, that is,
$$W_{xy}=\{w\in V(G):\ d_{G}(w,x) < d_{G}(w,y)\}\,.$$
Vertices $x$ and $y$ are {\em balanced} if $|W_{xy}| = |W_{yx}|$.  For an integer $\ell \in [\diam(G)] = \{1,2,\ldots, \diam(G)\}$ we say that $G$ is $\ell$-{\em distance-balanced} if each pair of vertices $x,y\in V(G)$ with $d_{G}(x,y) = \ell$ is balanced. $G$ is said to be {\em highly distance-balanced} if it is $\ell$-distance-balanced for every $\ell\in [\diam(G)]$. $1$-distance-balanced graphs are simply called {\em distance-balanced} graphs.

Distance-balanced graphs were first considered by Handa~\cite{Handa:1999} back in 1999, while the term ``distance-balanced'' was proposed a decade later by Jerebic et al.\ in~\cite{Jerebic:2008}. The latter paper was the trigger for intensive research of distance-balanced graphs, see~\cite{Abiad:2016, Balakrishnan:2014, Balakrishnan:2009, Cabello:2011, cavaleri-2020, fernardes-2022, Ilic:2010, Kutnar:2006, Kutnar:2009, Kutnar:2014, Miklavic:2012, YangR:2009}.
The study of distance-balanced graphs is interesting from various purely graph-theoretic aspects where one focuses on particular properties of such graphs such as symmetry, connectivity or complexity aspects of algorithms related to such graphs.
Moreover, distance-balanced graphs have motivated the introduction of the hitherto much-researched Mostar index~\cite{ali-2021, doslic-2018} and distance-unbalancedness of graphs~\cite{kr-2021, miklavic-2021, xu-2022}. In this context, distance-balanced graphs are the graphs with the Mostar index equal to 0.

In his dissertation~\cite{Frelih:2014}, Frelih generalized distance-balanced graphs to $\ell$-distance balanced graphs. The special case of $\ell=2$ has been studied in detail in~\cite{Frelih:2018}. Among other results it was demonstrated that there exist $2$-distance-balanced graphs that are not $1$-distance-balanced. $2$-distance-balanced graphs that are not $2$-connected were characterized as well as $2$-distance-balanced Cartesian and lexicographic products. In this direction, $\ell$-distance-balanced corona products and lexicographic products were investigated in~\cite{Jerebic:2021}. In~\cite{Miklavic:2018}, Miklavi\v{c} and \v{S}parl obtained some general results on $\ell$-distance balanced graphs. They studied graphs of diameter at most $3$ and investigated $\ell$-distance-balancedness of cubic graphs, in particular of generalized Petersen graphs.
Although generalized Petersen graphs are a family of cubic graphs but it is difficult to determine whether they are $\ell$-distance-balanced or not for some $\ell$.
And that is what has stimulated the main interest in this article. Let us define these graphs.

If $n\ge 3$ and $1\le k<n/2$, then the {\em generalized Petersen graph} $GP(n,k)$ is defined by
\begin{align*}
V(GP(n,k)) & = \{u_i:\ i\in \mathbb{Z}_n\} \cup\{v_i:\ i\in \mathbb{Z}_n\}, \\
E(GP(n,k)) & = \{u_iu_{i+1}:\ i\in \mathbb{Z}_n\} \cup\{v_iv_{i+k}:\ i\in \mathbb{Z}_n\} \cup \{u_iv_i:\ i\in \mathbb{Z}_n\}.
\end{align*}

In \cite{MaG:2023}, the authors proved that $GP(n,k)$ is $\diam(GP(n,k))$-distance-balanced
where $n$ is large relative to $k$. The following theorem was proved.
\begin{theorem}\label{T:GP-DDB}\cite{MaG:2023}
If $n$ and $k$ are integers, where $3\le k< n/2$ and
\begin{equation*}
n\ge\left\{\begin{array}{ll}
8; & k=3, \\
10; & k=4, \\
\frac{k(k+1)}{2}; & k \ \text{is odd and}\ k\ge 5,\\
\frac{k^2}{2}; & k \ \text{is even and}\ k\ge 6,
\end{array}\right.
\end{equation*}
then $GP(n,k)$ is $\diam(GP(n,k))$-distance-balanced.
\end{theorem}

In \cite{Miklavic:2018}, the authors gave the following conjecture and proved that
the conjecture was right when $k=2$.
\begin{conjecture}\cite{Miklavic:2018}\label{C:GP-onlyD-DB}
Let $k\ge 2$ be an integer and let
\begin{equation*}
n_k=\left\{\begin{array}{ll}
11, & k=2; \\
(k+1)^2, & k \ \text{odd};\\
k(k+2), & k\ge 4 \ \text{even}.
\end{array}\right.
\end{equation*}
Then for any $n>n_k$ the graph $GP(n,k)$ is not $\ell$-distance-balanced for any $1\le \ell<D$, where
$D$ is the diameter of $GP(n,k)$. Moreover, $n_k$ is the smallest integer with this property.
\end{conjecture}

In this paper, we study Conjecture~\ref{C:GP-onlyD-DB} and prove that Conjecture~\ref{C:GP-onlyD-DB} is right when $k=3,4$.
The following two theorems are the main results of the paper.
\begin{theorem}\label{T:GP(n,3)-onlyD-DB}
For any $n>16$, the generalized Petersen graph $GP(n,3)$ is not $\ell$-distance-balanced for any $1\le \ell<\diam(GP(n,3))$.
Moreover, $16$ is the smallest integer with this property.
\end{theorem}
\begin{theorem}\label{T:GP(n,4)-onlyD-DB}
For any $n>24$, the generalized Petersen graph $GP(n,4)$ is not $\ell$-distance-balanced for any $1\le \ell<\diam(GP(n,4))$.
Moreover, $24$ is the smallest integer with this property.
\end{theorem}

In section \ref{S:GP(n,3)}, we prove Theorem~\ref{T:GP(n,3)-onlyD-DB}.
In section \ref{S:GP(n,4)}, we prove Theorem~\ref{T:GP(n,4)-onlyD-DB}.
In section \ref{S:conluding}, we give one problem which is worth studying in the future.

\section{The proof of Theorem~\ref{T:GP(n,3)-onlyD-DB}}\label{S:GP(n,3)}

From \cite{Miklavic:2018}, the diameter of $GP(16,3)$ is $6$ and $GP(16,3)$ is $\ell$-distance-balanced
if and only if $\ell\in\{5,6\}$.
So we can suppose $n> 16$ when we prove Theorem \ref{T:GP(n,3)-onlyD-DB}. We will prove Theorem \ref{T:GP(n,3)-onlyD-DB}
via Proposition \ref{P:GP(n,3)-1}, \ref{P:GP(n,3)-2} and \ref{P:GP(n,3)-3}.

\begin{proposition}\label{P:GP(n,3)-1}
For any $n>16$, the generalized Petersen graph $GP(n,3)$ is not $1$-distance-balanced.
\end{proposition}
\begin{proof}
In $GP(n,3)$, $d(u_0,v_0)=1$ and we will prove that $|W_{u_0v_{0}}|<|W_{v_{0}u_0}|$. We divide the discussion into
the following six cases according to the size of $n$.

(1) When $n=6m$ where $m\ge 3$.

By symmetry, we just need to consider vertices $u_i$ and $v_i$ where $1\le i\le \frac{n}{2}$.

$d(u_0,v_{3t})=1+t$ and $d(v_0,v_{3t})=t$ where $1\le t\le m$.
$d(u_0,v_{3t+1})=2+t$ and $d(v_0,v_{3t+1})=3+t$ where $0\le t< m$.
$d(u_0,v_{3t+2})=3+t$ and $d(v_0,v_{3t+2})=4+t$ where $0\le t< m$.

$d(u_0,u_{3t})=2+t$ and $d(v_0,u_{3t})=1+t$ where $1\le t\le m$.
$d(u_0,u_{1})=1$ and $d(v_0,u_{1})=2$.
$d(u_0,u_{3t+1})=3+t$ and $d(v_0,u_{3t+1})=2+t$ where $1\le t< m$.
$d(u_0,u_{2})=2$ and $d(v_0,u_{2})=3$.
$d(u_0,u_{3t+2})=4+t$ and $d(v_0,u_{3t+2})=3+t$ where $1\le t< m$.

Note that $u_0\in W_{u_0v_{0}}$ and $v_{0}\in W_{v_{0}u_0}$.
Combined with the above discussion, $|W_{u_0v_{0}}|=2(2m+2)+1=4m+5$
and $|W_{v_{0}u_0}|=2(4m-4)+3=8m-5$. Because $m\ge 3$, $|W_{u_0v_{0}}|<|W_{v_{0}u_0}|$.

(2) When $n=6m+1$ where $m\ge 3$.

By symmetry, we just need to consider vertices $u_i$ and $v_i$ where $1\le i\le \frac{n}{2}$.

$d(u_0,v_{3t})=1+t$ and $d(v_0,v_{3t})=t$ where $1\le t\le m$.
$d(u_0,v_{3t+1})=2+t$ where $0\le t< m$.
$d(v_0,v_{3t+1})=3+t$ where $0\le t\le m-2$, and $d(v_0,v_{3(m-1)+1})=m+1$.
$d(u_0,v_{3t+2})=3+t$ and $d(v_0,v_{3t+2})=4+t$ where $0\le t< m$.

$d(u_0,u_{3t})=2+t$ and $d(v_0,u_{3t})=1+t$ where $1\le t\le m$.
$d(u_0,u_{1})=1$ and $d(v_0,u_{1})=2$.
$d(u_0,u_{3t+1})=3+t$ and $d(v_0,u_{3t+1})=2+t$ where $1\le t< m$.
$d(u_0,u_{2})=2$ and $d(v_0,u_{2})=3$.
$d(u_0,u_{3t+2})=4+t$ and $d(v_0,u_{3t+2})=3+t$ where $1\le t< m$.

Note that $u_0\in W_{u_0v_{0}}$ and $v_{0}\in W_{v_{0}u_0}$.
Combined with the above discussion, $|W_{u_0v_{0}}|=2(2m+1)+1=4m+3$
and $|W_{v_{0}u_0}|=2(4m-2)+1=8m-3$. Because $m\ge 3$, $|W_{u_0v_{0}}|<|W_{v_{0}u_0}|$.

(3) When $n=6m+2$ where $m\ge 3$.

By symmetry, we just need to consider vertices $u_i$ and $v_i$ where $1\le i\le \frac{n}{2}$.

$d(u_0,v_{3t})=1+t$ and $d(v_0,v_{3t})=t$ where $1\le t\le m$.
$d(u_0,v_{3t+1})=2+t$ and $d(v_0,v_{3t+1})=3+t$ where $0\le t\le m$.
$d(u_0,v_{3t+2})=3+t$ where $0\le t< m$.
$d(v_0,v_{3t+2})=4+t$ where $0\le t\le m-2$, and $d(v_0,v_{3(m-1)+2})=m+1$.

$d(u_0,u_{3t})=2+t$ and $d(v_0,u_{3t})=1+t$ where $1\le t\le m$.
$d(u_0,u_{1})=1$ and $d(v_0,u_{1})=2$.
$d(u_0,u_{3t+1})=3+t$ and $d(v_0,u_{3t+1})=2+t$ where $1\le t\le m$.
$d(u_0,u_{2})=2$ and $d(v_0,u_{2})=3$.
$d(u_0,u_{3t+2})=4+t$ and $d(v_0,u_{3t+2})=3+t$ where $1\le t< m$.

Note that $u_0\in W_{u_0v_{0}}$ and $v_{0}\in W_{v_{0}u_0}$.
Combined with the above discussion, $|W_{u_0v_{0}}|=2(2m+1)+2=4m+4$
and $|W_{v_{0}u_0}|=2(4m-1)+2=8m$. Because $m\ge 3$, $|W_{u_0v_{0}}|<|W_{v_{0}u_0}|$.

(4) When $n=6m+3$ where $m\ge 3$.

By symmetry, we just need to consider vertices $u_i$ and $v_i$ where $1\le i\le \frac{n}{2}$.

$d(u_0,v_{3t})=1+t$ and $d(v_0,v_{3t})=t$ where $1\le t\le m$.
$d(u_0,v_{3t+1})=2+t$ and $d(v_0,v_{3t+1})=3+t$ where $0\le t\le m$.
$d(u_0,v_{3t+2})=3+t$ and $d(v_0,v_{3t+2})=4+t$ where $0\le t< m$.

$d(u_0,u_{3t})=2+t$ and $d(v_0,u_{3t})=1+t$ where $1\le t\le m$.
$d(u_0,u_{1})=1$ and $d(v_0,u_{1})=2$.
$d(u_0,u_{3t+1})=3+t$ and $d(v_0,u_{3t+1})=2+t$ where $1\le t\le m$.
$d(u_0,u_{2})=2$ and $d(v_0,u_{2})=3$.
$d(u_0,u_{3t+2})=4+t$ and $d(v_0,u_{3t+2})=3+t$ where $1\le t< m$.

Note that $u_0\in W_{u_0v_{0}}$ and $v_{0}\in W_{v_{0}u_0}$.
Combined with the above discussion, $|W_{u_0v_{0}}|=2(2m+3)+1=4m+7$
and $|W_{v_{0}u_0}|=2(4m-1)+1=8m-1$. Because $m\ge 3$, $|W_{u_0v_{0}}|<|W_{v_{0}u_0}|$.

(5) When $n=6m+4$ where $m\ge 3$.

By symmetry, we just need to consider vertices $u_i$ and $v_i$ where $1\le i\le \frac{n}{2}$.

$d(u_0,v_{3t})=1+t$ and $d(v_0,v_{3t})=t$ where $1\le t\le m$.
$d(u_0,v_{3t+1})=2+t$ where $0\le t\le m$.
$d(v_0,v_{3t+1})=3+t$ where $0\le t\le m-1$, and $d(v_0,v_{3m+1})=m+1$.
$d(u_0,v_{3t+2})=3+t$ and $d(v_0,v_{3t+2})=4+t$ where $0\le t\le m$.

$d(u_0,u_{3t})=2+t$ and $d(v_0,u_{3t})=1+t$ where $1\le t\le m$.
$d(u_0,u_{1})=1$ and $d(v_0,u_{1})=2$.
$d(u_0,u_{3t+1})=3+t$ and $d(v_0,u_{3t+1})=2+t$ where $1\le t\le m$.
$d(u_0,u_{2})=2$ and $d(v_0,u_{2})=3$.
$d(u_0,u_{3t+2})=4+t$ and $d(v_0,u_{3t+2})=3+t$ where $1\le t\le m$.

Note that $u_0\in W_{u_0v_{0}}$ and $v_{0}\in W_{v_{0}u_0}$.
Combined with the above discussion, $|W_{u_0v_{0}}|=2(2m+2)+2=4m+6$
and $|W_{v_{0}u_0}|=2\times 4m+2=8m+2$. Because $m\ge 3$, $|W_{u_0v_{0}}|<|W_{v_{0}u_0}|$.

(6) When $n=6m+5$ where $m\ge 2$.

By symmetry, we just need to consider vertices $u_i$ and $v_i$ where $1\le i\le \frac{n}{2}$.

$d(u_0,v_{3t})=1+t$ and $d(v_0,v_{3t})=t$ where $1\le t\le m$.
$d(u_0,v_{3t+1})=2+t$ and $d(v_0,v_{3t+1})=3+t$ where $0\le t\le m$.
$d(u_0,v_{3t+2})=3+t$ where $0\le t\le m-1$ and $d(u_0,v_{3m+2})=m+2$.
$d(v_0,v_{3t+2})=4+t$ where $0\le t\le m-2$, $d(v_0,v_{3(m-1)+2})=m+2$ and $d(v_0,v_{3m+2})=m+1$.

$d(u_0,u_{3t})=2+t$ and $d(v_0,u_{3t})=1+t$ where $1\le t\le m$.
$d(u_0,u_{1})=1$ and $d(v_0,u_{1})=2$.
$d(u_0,u_{3t+1})=3+t$ and $d(v_0,u_{3t+1})=2+t$ where $1\le t\le m$.
$d(u_0,u_{2})=2$ and $d(v_0,u_{2})=3$.
$d(u_0,u_{3t+2})=4+t$ and $d(v_0,u_{3t+2})=3+t$ where $1\le t\le m-1$.
$d(u_0,u_{3m+2})=m+3$ and $d(v_0,u_{3m+2})=m+2$.

Note that $u_0\in W_{u_0v_{0}}$ and $v_{0}\in W_{v_{0}u_0}$.
Combined with the above discussion, $|W_{u_0v_{0}}|=2(2m+2)+1=4m+5$
and $|W_{v_{0}u_0}|=2(4m+1)+1=8m+3$. Because $m\ge 2$, $|W_{u_0v_{0}}|<|W_{v_{0}u_0}|$.

\end{proof}

\begin{proposition}\label{P:GP(n,3)-2}
For any $n>16$, the generalized Petersen graph $GP(n,3)$ is not $2$-distance-balanced.
\end{proposition}
\begin{proof}
In $GP(n,3)$, $d(u_0,v_{-3})=2$ and we will prove that $|W_{u_0v_{-3}}|<|W_{v_{-3}u_0}|$.
Note that $v_{-3}=v_{n-3}$.

Firstly we consider vertices $v_{-1},v_{-2},u_{-1},u_{-2}$.

$d(u_0,v_{-1})=2$ and $d(v_{-3},v_{-1})=4$. $d(u_0,v_{-2})=d(v_{-3},v_{-2})=3$.
$d(u_0,u_{-1})=1$ and $d(v_{-3},u_{-1})=3$. $d(u_0,u_{-2})=d(v_{-3},u_{-2})=2$.
So $u_{-1},v_{-1}\in W_{u_0v_{-3}}$ and no vertex of $\{v_{-1},v_{-2},u_{-1},u_{-2}\}$ is in $W_{v_{-3}u_0}$.

Next we consider vertices $v_i$ where $0\le i<n-3$ and $u_j$ where $1\le j\le n-3$.
We divide the discussion into
the following six cases according to the size of $n$.

(1) When $n=6m$ where $m\ge 3$.

Note that $n-3=6m-3=3(2m-1)$.

$d(u_0,v_{3t})=d(v_{6m-3},v_{3t})=1+t$ when $0\le t\le m-1$.
$d(v_{6m-3},v_{3t})=2m-1-t$ and $d(u_0,v_{3t})>d(v_{6m-3},v_{3t})$ when $m\le t<2m-1$.
$d(u_0,v_{3t+1})=2+t$ and $d(u_0,v_{3t+1})<d(v_{6m-3},v_{3t+1})$ when $0\le t\le m-1$.
$d(u_0,v_{3t+1})=d(v_{6m-3},v_{3t+1})=2m-t+2$ when $m\le t< 2m-1$.
$d(u_0,v_{3t+2})=3+t$ and $d(u_0,v_{3t+2})<d(v_{6m-3},v_{3t+2})$ when $0\le t\le m-2$.
$d(u_0,v_{3t+2})=d(v_{6m-3},v_{3t+2})=2m-t+1$ when $m-1\le t< 2m-1$.

$d(u_0,u_{3t})=d(v_{6m-3},u_{3t})=2+t$ when $1\le t\le m-1$.
$d(v_{6m-3},u_{3t})=2m-t$ and $d(u_0,u_{3t})>d(v_{6m-3},u_{3t})$ when $m\le t\le 2m-1$.
$d(u_0,u_{1})=1$ and $d(v_{6m-3},u_{1})=2m+1$.
$d(u_0,u_{3t+1})=d(v_{6m-3},u_{3t+1})=3+t$ when $1\le t\le m-1$.
$d(v_{6m-3},u_{3t+1})=2m-t+1$ and $d(u_0,u_{3t+1})>d(v_{6m-3},u_{3t+1})$ when $m\le t< 2m-1$.
$d(u_0,u_{2})=2$ and $d(v_{6m-3},u_{2})=2m$.
$d(u_0,u_{3t+2})=d(v_{6m-3},u_{3t+2})=4+t$ when $1\le t\le m-2$.
$d(v_{6m-3},u_{3t+2})=2m-t$ and $d(u_0,u_{3t+2})>d(v_{6m-3},u_{3t+2})$ when $m-1\le t< 2m-1$.

Note that $u_0\in W_{u_0v_{6m-3}}$ and $v_{6m-3}\in W_{v_{6m-3}u_0}$.
Combined with the above discussion, $|W_{u_0v_{6m-3}}|=2m+4$
and $|W_{v_{6m-3}u_0}|=4m-1$. Because $m\ge 3$, $|W_{u_0v_{6m-3}}|<|W_{v_{6m-3}u_0}|$.

(2) When $n=6m+1$ where $m\ge 3$.

Note that $n-3=6m-2=3(2m-1)+1$.

$d(u_0,v_{3t})=d(v_{6m-2},v_{3t})=1+t$ when $0\le t\le m$.
$d(u_0,v_{3t})=d(v_{6m-2},v_{3t})=2m-t+2$ when $m+1\le t\le 2m-1$.
$d(u_0,v_{3t+1})=2+t$ and $d(u_0,v_{3t+1})<d(v_{6m-2},v_{3t+1})$ when $0\le t\le m-2$.
$d(v_{6m-2},v_{3t+1})=2m-t-1$ and $d(u_0,v_{3t+1})>d(v_{6m-2},v_{3t+1})$ when $m-1\le t< 2m-1$.
$d(u_0,v_{3t+2})=3+t$ and $d(u_0,v_{3t+2})<d(v_{6m-2},v_{3t+2})$ when $0\le t\le m-1$.
$d(u_0,v_{3t+2})=d(v_{6m-2},v_{3t+2})=2m-t+2$ when $m\le t< 2m-1$.

$d(u_0,u_{3t})=d(v_{6m-2},u_{3t})=2+t$ when $1\le t\le m-1$.
$d(v_{6m-2},u_{3t})=2m-t+1$ and $d(u_0,u_{3t})>d(v_{6m-2},u_{3t})$ when $m\le t\le 2m-1$.
$d(u_0,u_{1})=1$ and $d(v_{6m-2},u_{1})=2m$.
$d(u_0,u_{3t+1})=d(v_{6m-2},u_{3t+1})=3+t$ when $1\le t\le m-2$.
$d(v_{6m-2},u_{3t+1})=2m-t$ and $d(u_0,u_{3t+1})>d(v_{6m-2},u_{3t+1})$ when $m-1\le t\le 2m-1$.
$d(u_0,u_{2})=2$ and $d(v_{6m-2},u_{2})=2m+1$.
$d(u_0,u_{3t+2})=d(v_{6m-2},u_{3t+2})=4+t$ when $1\le t\le m-2$.
$d(v_{6m-2},u_{3t+2})=2m-t+1$ and $d(u_0,u_{3t+2})>d(v_{6m-2},u_{3t+2})$ when $m-1\le t< 2m-1$.

Note that $u_0\in W_{u_0v_{6m-2}}$ and $v_{6m-2}\in W_{v_{6m-2}u_0}$.
Combined with the above discussion, $|W_{u_0v_{6m-2}}|=2m+4$
and $|W_{v_{6m-2}u_0}|=4m+2$. Because $m\ge 3$, $|W_{u_0v_{6m-2}}|<|W_{v_{6m-2}u_0}|$.

(3) When $n=6m+2$ where $m\ge 3$.

Note that $n-3=6m-1=3(2m-1)+2$.

$d(u_0,v_{3t})=d(v_{6m-1},v_{3t})=1+t$ when $0\le t\le m+1$.
$d(u_0,v_{3t})=d(v_{6m-1},v_{3t})=2m-t+3$ when $m+2\le t\le 2m-1$.
$d(u_0,v_{3t+1})=2+t$ and $d(u_0,v_{3t+1})<d(v_{6m-1},v_{3t+1})$ when $0\le t\le m-1$.
$d(u_0,v_{3t+1})=d(v_{6m-1},v_{3t+1})=2m-t+2$ when $m\le t\le 2m-1$.
$d(u_0,v_{3t+2})=3+t$ and $d(u_0,v_{3t+2})<d(v_{6m-1},v_{3t+2})$ when $0\le t\le m-3$.
$d(u_0,v_{3(m-2)+2})=d(v_{6m-1},v_{3(m-2)+2})=m+1$.
$d(v_{6m-1},v_{3t+2})=2m-t-1$ and $d(u_0,v_{3t+2})>d(v_{6m-1},v_{3t+2})$ when $m-1\le t< 2m-1$.

$d(u_0,u_{3t})=d(v_{6m-1},u_{3t})=2+t$ when $1\le t\le m$.
$d(v_{6m-1},u_{3t})=2m-t+2$ and $d(u_0,u_{3t})>d(v_{6m-1},u_{3t})$ when $m+1\le t\le 2m-1$.
$d(u_0,u_{1})=1$ and $d(v_{6m-1},u_{1})=2m+1$.
$d(u_0,u_{3t+1})=d(v_{6m-1},u_{3t+1})=3+t$ when $1\le t\le m-1$.
$d(v_{6m-1},u_{3t+1})=2m-t+1$ and $d(u_0,u_{3t+1})>d(v_{6m-1},u_{3t+1})$ when $m\le t\le 2m-1$.
$d(u_0,u_{2})=2$ and $d(v_{6m-1},u_{2})=2m$.
$d(u_0,u_{3t+2})=d(v_{6m-1},u_{3t+2})=4+t$ when $1\le t\le m-2$.
$d(v_{6m-1},u_{3t+2})=2m-t$ and $d(u_0,u_{3t+2})>d(v_{6m-1},u_{3t+2})$ when $m-1\le t\le 2m-1$.

Note that $u_0\in W_{u_0v_{6m-1}}$ and $v_{6m-1}\in W_{v_{6m-1}u_0}$.
Combined with the above discussion, $|W_{u_0v_{6m-1}}|=2m+3$
and $|W_{v_{6m-1}u_0}|=4m+1$. Because $m\ge 3$, $|W_{u_0v_{6m-1}}|<|W_{v_{6m-1}u_0}|$.

(4) When $n=6m+3$ where $m\ge 3$.

Note that $n-3=6m=3\times 2m$.

$d(u_0,v_{3t})=d(v_{6m},v_{3t})=1+t$ when $0\le t\le m-1$.
$d(v_{6m},v_{3t})=2m-t$ and $d(u_0,v_{3t})>d(v_{6m},v_{3t})$ when $m\le t< 2m$.
$d(u_0,v_{3t+1})=2+t$ and $d(u_0,v_{3t+1})<d(v_{6m},v_{3t+1})$ when $0\le t\le m$.
$d(u_0,v_{3t+1})=d(v_{6m},v_{3t+1})=2m-t+3$ when $m+1\le t< 2m$.
$d(u_0,v_{3t+2})=3+t$ and $d(u_0,v_{3t+2})<d(v_{6m},v_{3t+2})$ when $0\le t\le m-1$.
$d(u_0,v_{3t+2})=d(v_{6m},v_{3t+2})=2m-t+2$ when $m\le t< 2m$.

$d(u_0,u_{3t})=d(v_{6m},u_{3t})=2+t$ when $1\le t\le m-1$.
$d(v_{6m},u_{3t})=2m-t+1$ and $d(u_0,u_{3t})>d(v_{6m},u_{3t})$ when $m\le t\le 2m$.
$d(u_0,u_{1})=1$ and $d(v_{6m},u_{1})=2m+2$.
$d(u_0,u_{3t+1})=d(v_{6m},u_{3t+1})=3+t$ when $1\le t\le m-1$.
$d(v_{6m},u_{3t+1})=2m-t+2$ and $d(u_0,u_{3t+1})>d(v_{6m},u_{3t+1})$ when $m\le t< 2m$.
$d(u_0,u_{2})=2$ and $d(v_{6m},u_{2})=2m+1$.
$d(u_0,u_{3t+2})=d(v_{6m},u_{3t+2})=4+t$ when $1\le t\le m-2$.
$d(v_{6m},u_{3t+2})=2m-t+1$ and $d(u_0,u_{3t+2})>d(v_{6m},u_{3t+2})$ when $m-1\le t< 2m$.

Note that $u_0\in W_{u_0v_{6m}}$ and $v_{6m}\in W_{v_{6m}u_0}$.
Combined with the above discussion, $|W_{u_0v_{6m}}|=2m+6$
and $|W_{v_{6m}u_0}|=4m+3$. Because $m\ge 3$, $|W_{u_0v_{6m}}|<|W_{v_{6m}u_0}|$.

(5) When $n=6m+4$ where $m\ge 3$.

Note that $n-3=6m+1=3\times 2m+1$.

$d(u_0,v_{3t})=d(v_{6m+1},v_{3t})=1+t$ when $0\le t\le m+1$.
$d(u_0,v_{3t})=d(v_{6m+1},v_{3t})=2m-t+3$ when $m+2\le t\le 2m$.
$d(u_0,v_{3t+1})=2+t$ and $d(u_0,v_{3t+1})<d(v_{6m+1},v_{3t+1})$ when $0\le t\le m-2$.
$d(u_0,v_{3(m-1)+1})=d(v_{6m+1},v_{3(m-1)+1})=m+1$.
$d(v_{6m+1},v_{3t+1})=2m-t$ and $d(u_0,v_{3t+1})>d(v_{6m+1},v_{3t+1})$ when $m\le t< 2m$.
$d(u_0,v_{3t+2})=3+t$ and $d(u_0,v_{3t+2})<d(v_{6m+1},v_{3t+2})$ when $0\le t\le m-1$.
$d(u_0,v_{3t+2})=d(v_{6m+1},v_{3t+2})=2m-t+3$ when $m\le t< 2m$.

$d(u_0,u_{3t})=d(v_{6m+1},u_{3t})=2+t$ when $1\le t\le m$.
$d(v_{6m+1},u_{3t})=2m-t+2$ and $d(u_0,u_{3t})>d(v_{6m+1},u_{3t})$ when $m+1\le t\le 2m$.
$d(u_0,u_{1})=1$ and $d(v_{6m+1},u_{1})=2m+1$.
$d(u_0,u_{3t+1})=d(v_{6m+1},u_{3t+1})=3+t$ when $1\le t\le m-1$.
$d(v_{6m+1},u_{3t+1})=2m-t+1$ and $d(u_0,u_{3t+1})>d(v_{6m+1},u_{3t+1})$ when $m\le t\le 2m$.
$d(u_0,u_{2})=2$ and $d(v_{6m},u_{2})=2m+2$.
$d(u_0,u_{3t+2})=d(v_{6m+1},u_{3t+2})=4+t$ when $1\le t\le m-1$.
$d(v_{6m+1},u_{3t+2})=2m-t+2$ and $d(u_0,u_{3t+2})>d(v_{6m+1},u_{3t+2})$ when $m\le t< 2m$.

Note that $u_0\in W_{u_0v_{6m+1}}$ and $v_{6m+1}\in W_{v_{6m+1}u_0}$.
Combined with the above discussion, $|W_{u_0v_{6m+1}}|=2m+4$
and $|W_{v_{6m+1}u_0}|=4m+2$. Because $m\ge 3$, $|W_{u_0v_{6m+1}}|<|W_{v_{6m+1}u_0}|$.

(6) When $n=6m+5$ where $m\ge 2$.

Note that $n-3=6m+2=3\times 2m+2$.

$d(u_0,v_{3t})=d(v_{6m+2},v_{3t})=1+t$ when $0\le t\le m+1$.
$d(u_0,v_{3t})=d(v_{6m+2},v_{3t})=2m-t+4$ when $m+2\le t\le 2m$.
$d(u_0,v_{3t+1})=2+t$ and $d(u_0,v_{3t+1})<d(v_{6m+2},v_{3t+1})$ when $0\le t\le m$.
$d(u_0,v_{3t+1})=d(v_{6m+2},v_{3t+1})=2m-t+3$ when $m+1\le t\le 2m$.
$d(u_0,v_{3t+2})=3+t$ and $d(u_0,v_{3t+2})<d(v_{6m+2},v_{3t+2})$ when $0\le t\le m-2$.
$d(v_{6m+2},v_{3t+2})=2m-t$ and $d(u_0,v_{3t+2})>d(v_{6m+1},v_{3t+2})$ when $m-1\le t< 2m$.

$d(u_0,u_{3t})=d(v_{6m+2},u_{3t})=2+t$ when $1\le t\le m$.
$d(v_{6m+2},u_{3t})=2m-t+3$ and $d(u_0,u_{3t})>d(v_{6m+2},u_{3t})$ when $m+1\le t\le 2m$.
$d(u_0,u_{1})=1$ and $d(v_{6m+2},u_{1})=2m+2$.
$d(u_0,u_{3t+1})=d(v_{6m+2},u_{3t+1})=3+t$ when $1\le t\le m-1$.
$d(v_{6m+2},u_{3t+1})=2m-t+2$ and $d(u_0,u_{3t+1})>d(v_{6m+2},u_{3t+1})$ when $m\le t\le 2m$.
$d(u_0,u_{2})=2$ and $d(v_{6m},u_{2})=2m+1$.
$d(u_0,u_{3t+2})=d(v_{6m+2},u_{3t+2})=4+t$ when $1\le t\le m-2$.
$d(v_{6m+2},u_{3t+2})=2m-t+1$ and $d(u_0,u_{3t+2})>d(v_{6m+2},u_{3t+2})$ when $m-1\le t\le 2m$.

Note that $u_0\in W_{u_0v_{6m+2}}$ and $v_{6m+2}\in W_{v_{6m+2}u_0}$.
Combined with the above discussion, $|W_{u_0v_{6m+2}}|=2m+5$
and $|W_{v_{6m+2}u_0}|=4m+5$. Because $m\ge 2$, $|W_{u_0v_{6m+2}}|<|W_{v_{6m+2}u_0}|$.

\end{proof}

\begin{proposition}\label{P:GP(n,3)-3}
For any $n>16$, the generalized Petersen graph $GP(n,3)$ is not $\ell$-distance-balanced for any $3\le \ell<\diam(GP(n,3))$.
\end{proposition}
\begin{proof}
For any $3\le\ell<D$, we first show that there exists $v_j$ such that $d(u_0,v_j)=\ell$ where $6\le j\le n/2$.
From \cite{MaG:2023}, there exists $j^*$ such that $d(u_0,u_{j^*})=D$.

When $n=6m$ $(m\ge 3)$ or $n=6m+1$ $(m\ge 3)$, from \cite{MaG:2023} we know that $j^*=3(m-1)+2$ and $D=d(u_0,u_{j^*})=m+3$.
Note that $d(u_0,v_{3s+2})=s+3$ where $2\le s\le m-1$ and $d(u_0,v_{3s})=s+1$ where $2\le s\le m$.

When $n=6m+2$ $(m\ge 3)$ or $n=6m+3$ $(m\ge 3)$, from \cite{MaG:2023} we know that $j^*=3m+1$ and $D=d(u_0,u_{j^*})=m+3$.
Note that $d(u_0,v_{3s+1})=s+2$ where $2\le s\le m$ and $d(u_0,v_{3s})=s+1$ where $2\le s\le m$.

When $n=6m+4$ $(m\ge 3)$, from \cite{MaG:2023} we know that $j^*=3m+2$ and $D=d(u_0,u_{j^*})=m+4$.
Note that $d(u_0,v_{3s+2})=s+3$ where $2\le s\le m$ and $d(u_0,v_{3s})=s+1$ where $2\le s\le m$.

When $n=6m+5$ $(m\ge 2)$, from \cite{MaG:2023} we know that $j^*=3m+1$ and $D=d(u_0,u_{j^*})=m+3$.
Note that $d(u_0,v_{3s+1})=s+2$ where $2\le s\le m$ and $d(u_0,v_{3s})=s+1$ where $2\le s\le m$.

From the above discussion, there exists $j$ where $6\le j\le n/2$ such that $d(u_0,v_j)=\ell$ for any $3\le\ell<D$.
Let $V_1=\{u_i\mid 1\le i\le j-1\}\cup\{v_i\mid 1\le i\le j-1\}$ and $V_2=\{u_i\mid j+1\le i\le n-1\}\cup\{v_i\mid j+1\le i\le n-1\}$.
Let $W^1_{u_0v_j}$ be the set of vertices which are in $W_{u_0v_j}$ and also in $V_1\cup\{u_0,v_0,u_j,v_j\}$.
Let $W^1_{v_ju_0}$ be the set of vertices which are in $W_{v_ju_0}$ and also in $V_1\cup\{u_0,v_0,u_j,v_j\}$.
Let $W^2_{u_0v_j}$ be the set of vertices which are in $W_{u_0v_j}$ and also in $V_2\cup\{u_0,v_0,u_j,v_j\}$.
Let $W^2_{v_ju_0}$ be the set of vertices which are in $W_{v_ju_0}$ and also in $V_2\cup\{u_0,v_0,u_j,v_j\}$.
Because $6\le j\le n/2$, $|W^2_{u_0v_j}|=|W^1_{u_0v_{n-j}}|$ and $|W^2_{v_ju_0}|=|W^1_{v_{n-j}u_0}|$. So
$|W_{u_0v_j}|=|W^1_{u_0v_j}|+|W^2_{u_0v_j}|-2=|W^1_{u_0v_j}|+|W^1_{u_0v_{n-j}}|-2$ and
$|W_{v_ju_0}|=|W^1_{v_ju_0}|+|W^2_{v_ju_0}|-2=|W^1_{v_ju_0}|+|W^1_{v_{n-j}u_0}|-2$.
In the following we will compute $|W^1_{u_0v_j}|$ and $|W^1_{v_ju_0}|$ where $6\le j\le n-6$.
The discussion is divided into the following six cases.

(1a) Computation of $|W^1_{u_0v_{3s}}|$ and $|W^1_{v_{3s}u_0}|$ when $s$ is odd and $s\ge 3$.

When $s=3$,

$d(u_0,v_0)=1$ and $d(v_9,v_0)=3$. $d(u_0,v_3)=d(v_9,v_3)=2$. $d(u_0,v_6)=3$ and $d(v_9,v_6)=1$.
$d(u_0,v_1)=2$ and $d(v_9,v_1)=6$. $d(u_0,v_4)=3$ and $d(v_9,v_4)=5$. $d(u_0,v_7)=4$ and $d(v_9,v_7)=4$.
$d(u_0,v_2)=3$ and $d(v_9,v_2)=5$. $d(u_0,v_5)=4$ and $d(v_9,v_5)=4$. $d(u_0,v_8)=5$ and $d(v_9,v_8)=3$.
So $v_0,v_1,v_4,v_2\in W^1_{u_0v_{9}}$ and $v_6,v_8\in W^1_{v_{9}u_0}$.

$d(u_0,u_3)=3$ and $d(v_9,u_3)=3$. $d(u_0,u_6)=4$ and $d(v_9,u_6)=2$. $d(u_0,u_9)=5$ and $d(v_9,u_9)=1$.
$d(u_0,u_1)=1$ and $d(v_9,u_1)=5$. $d(u_0,u_4)=4$ and $d(v_9,u_4)=4$. $d(u_0,u_7)=5$ and $d(v_9,u_7)=3$.
$d(u_0,u_2)=2$ and $d(v_9,u_2)=4$. $d(u_0,u_5)=5$ and $d(v_9,u_5)=3$. $d(u_0,u_8)=6$ and $d(v_9,u_8)=2$.
So $u_1,u_2\in W^1_{u_0v_{9}}$ and $u_6,u_9,u_7,u_5,u_8\in W^1_{v_{9}u_0}$.

Note that $u_0\in W^1_{u_0v_{9}}$ and $v_{9}\in W^1_{v_{9}u_0}$.
Combined with the above discussion, $|W^1_{u_0v_{9}}|=7$ and $|W^1_{v_{9}u_0}|=8$.

When $s\ge 5$,

$d(u_0,v_{3t})=1+t$ and $d(v_{3s},v_{3t})=s-t$ where $0\le t<s$.
When $0\le t<\frac{s-1}{2}$, $d(u_0,v_{3t})<d(v_{3s},v_{3t})$.
When $\frac{s-1}{2}<t<s$, $d(u_0,v_{3t})>d(v_{3s},v_{3t})$.
$d(u_0,v_{3t+1})=2+t$ and $d(v_{3s},v_{3t+1})=s-t+3$ where $0\le t<s$.
When $0\le t<\frac{s+1}{2}$, $d(u_0,v_{3t+1})<d(v_{3s},v_{3t+1})$.
When $\frac{s+1}{2}<t<s$, $d(u_0,v_{3t+1})>d(v_{3s},v_{3t+1})$.
$d(u_0,v_{3t+2})=3+t$ and $d(v_{3s},v_{3t+2})=s-t+2$ where $0\le t<s$.
When $0\le t<\frac{s-1}{2}$, $d(u_0,v_{3t+2})<d(v_{3s},v_{3t+2})$.
When $\frac{s-1}{2}<t<s$, $d(u_0,v_{3t+2})>d(v_{3s},v_{3t+2})$.

$d(u_0,u_{3t})=2+t$ and $d(v_{3s},u_{3t})=s-t+1$ where $1\le t\le s$.
When $1\le t<\frac{s-1}{2}$, $d(u_0,u_{3t})<d(v_{3s},u_{3t})$.
When $\frac{s-1}{2}<t\le s$, $d(u_0,u_{3t})>d(v_{3s},u_{3t})$.
$d(u_0,u_{1})=1$ and $d(v_{3s},u_{1})=s+2$.
$d(u_0,u_{3t+1})=3+t$ and $d(v_{3s},u_{3t+1})=s-t+2$ where $1\le t< s$.
When $1\le t<\frac{s-1}{2}$, $d(u_0,u_{3t+1})<d(v_{3s},u_{3t+1})$.
When $\frac{s-1}{2}<t< s$, $d(u_0,u_{3t+1})>d(v_{3s},u_{3t+1})$.
$d(u_0,u_{2})=2$ and $d(v_{3s},u_{2})=s+1$.
$d(u_0,u_{3t+2})=4+t$ and $d(v_{3s},u_{3t+2})=s-t+1$ where $1\le t< s$.
When $1\le t<\frac{s-3}{2}$, $d(u_0,u_{3t+2})<d(v_{3s},u_{3t+2})$.
When $\frac{s-3}{2}<t< s$, $d(u_0,u_{3t+2})>d(v_{3s},u_{3t+2})$.

Note that $u_0\in W^1_{u_0v_{3s}}$ and $v_{3s}\in W^1_{v_{3s}u_0}$.
Combined with the above discussion, $|W^1_{u_0v_{3s}}|=3s-3$ and $|W^1_{v_{3s}u_0}|=3s-1$.

(1b) Computation of $|W^1_{u_0v_{3s}}|$ and $|W^1_{v_{3s}u_0}|$ when $s$ is even and $s\ge 2$.

When $s=2$,

$d(u_0,v_0)=1$ and $d(v_6,v_0)=2$. $d(u_0,v_3)=2$ and $d(v_6,v_3)=1$.
$d(u_0,v_1)=2$ and $d(v_6,v_1)=5$. $d(u_0,v_4)=3$ and $d(v_6,v_4)=4$.
$d(u_0,v_2)=3$ and $d(v_6,v_2)=4$. $d(u_0,v_5)=4$ and $d(v_6,v_5)=3$.
So $v_0,v_1,v_2,v_4\in W^1_{u_0v_{6}}$ and $v_3,v_5\in W^1_{v_{6}u_0}$.

$d(u_0,u_3)=3$ and $d(v_6,u_3)=2$. $d(u_0,u_6)=4$ and $d(v_6,u_6)=1$.
$d(u_0,u_1)=1$ and $d(v_6,u_1)=4$. $d(u_0,u_4)=4$ and $d(v_6,u_4)=3$.
$d(u_0,u_2)=2$ and $d(v_6,u_2)=3$. $d(u_0,u_5)=5$ and $d(v_6,u_5)=2$.
So $u_1,u_2\in W^1_{u_0v_{6}}$ and $u_3,u_4,u_5,u_6\in W^1_{v_{6}u_0}$.

Note that $u_0\in W^1_{u_0v_{6}}$ and $v_{6}\in W^1_{v_{6}u_0}$.
Combined with the above discussion, $|W^1_{u_0v_{6}}|=7$ and $|W^1_{v_{6}u_0}|=7$.

When $s\ge 4$,

$d(u_0,v_{3t})=1+t$ and $d(v_{3s},v_{3t})=s-t$ where $0\le t<s$.
When $0\le t\le\frac{s-2}{2}$, $d(u_0,v_{3t})<d(v_{3s},v_{3t})$.
When $\frac{s}{2}\le t<s$, $d(u_0,v_{3t})>d(v_{3s},v_{3t})$.
$d(u_0,v_{3t+1})=2+t$ and $d(v_{3s},v_{3t+1})=s-t+3$ where $0\le t<s$.
When $0\le t\le\frac{s}{2}$, $d(u_0,v_{3t+1})<d(v_{3s},v_{3t+1})$.
When $\frac{s+2}{2}\le t<s$, $d(u_0,v_{3t+1})>d(v_{3s},v_{3t+1})$.
$d(u_0,v_{3t+2})=3+t$ and $d(v_{3s},v_{3t+2})=s-t+2$ where $0\le t<s$.
When $0\le t\le\frac{s-2}{2}$, $d(u_0,v_{3t+2})<d(v_{3s},v_{3t+2})$.
When $\frac{s}{2}\le t<s$, $d(u_0,v_{3t+2})>d(v_{3s},v_{3t+2})$.

$d(u_0,u_{3t})=2+t$ and $d(v_{3s},u_{3t})=s-t+1$ where $1\le t\le s$.
When $1\le t\le\frac{s-2}{2}$, $d(u_0,u_{3t})<d(v_{3s},u_{3t})$.
When $\frac{s}{2}\le t\le s$, $d(u_0,u_{3t})>d(v_{3s},u_{3t})$.
$d(u_0,u_{1})=1$ and $d(v_{3s},u_{1})=s+2$.
$d(u_0,u_{3t+1})=3+t$ and $d(v_{3s},u_{3t+1})=s-t+2$ where $1\le t< s$.
When $1\le t\le\frac{s-2}{2}$, $d(u_0,u_{3t+1})<d(v_{3s},u_{3t+1})$.
When $\frac{s}{2}\le t< s$, $d(u_0,u_{3t+1})>d(v_{3s},u_{3t+1})$.
$d(u_0,u_{2})=2$ and $d(v_{3s},u_{2})=s+1$.
$d(u_0,u_{3t+2})=4+t$ and $d(v_{3s},u_{3t+2})=s-t+1$ where $1\le t< s$.
When $1\le t\le\frac{s-4}{2}$, $d(u_0,u_{3t+2})<d(v_{3s},u_{3t+2})$.
When $\frac{s-2}{2}\le t< s$, $d(u_0,u_{3t+2})>d(v_{3s},u_{3t+2})$.

Note that $u_0\in W^1_{u_0v_{3s}}$ and $v_{3s}\in W^1_{v_{3s}u_0}$.
Combined with the above discussion, $|W^1_{u_0v_{3s}}|=3s$ and $|W^1_{v_{3s}u_0}|=3s+2$.

(2a) Computation of $|W^1_{u_0v_{3s+1}}|$ and $|W^1_{v_{3s+1}u_0}|$ when $s$ is odd and $s\ge 3$.

$d(u_0,v_{3t})=1+t$ and $d(v_{3s+1},v_{3t})=s-t+3$ where $0\le t\le s$.
When $0\le t\le\frac{s+1}{2}$, $d(u_0,v_{3t})<d(v_{3s+1},v_{3t})$.
When $\frac{s+3}{2}\le t\le s$, $d(u_0,v_{3t})>d(v_{3s+1},v_{3t})$.
$d(u_0,v_{3t+1})=2+t$ and $d(v_{3s+1},v_{3t+1})=s-t$ where $0\le t<s$.
When $0\le t\le\frac{s-3}{2}$, $d(u_0,v_{3t+1})<d(v_{3s+1},v_{3t+1})$.
When $\frac{s-1}{2}\le t<s$, $d(u_0,v_{3t+1})>d(v_{3s+1},v_{3t+1})$.
$d(u_0,v_{3t+2})=3+t$ and $d(v_{3s+1},v_{3t+2})=s-t+3$ where $0\le t<s$.
When $0\le t\le\frac{s-1}{2}$, $d(u_0,v_{3t+2})<d(v_{3s+1},v_{3t+2})$.
When $\frac{s+1}{2}\le t<s$, $d(u_0,v_{3t+2})>d(v_{3s+1},v_{3t+2})$.

$d(u_0,u_{3t})=2+t$ and $d(v_{3s+1},u_{3t})=s-t+2$ where $1\le t\le s$.
When $1\le t\le\frac{s-1}{2}$, $d(u_0,u_{3t})<d(v_{3s+1},u_{3t})$.
When $\frac{s+1}{2}\le t\le s$, $d(u_0,u_{3t})>d(v_{3s+1},u_{3t})$.
$d(u_0,u_{1})=1$ and $d(v_{3s+1},u_{1})=s+1$.
$d(u_0,u_{3t+1})=3+t$ and $d(v_{3s+1},u_{3t+1})=s-t+1$ where $1\le t\le s$.
When $1\le t\le\frac{s-3}{2}$, $d(u_0,u_{3t+1})<d(v_{3s+1},u_{3t+1})$.
When $\frac{s-1}{2}\le t\le s$, $d(u_0,u_{3t+1})>d(v_{3s+1},u_{3t+1})$.
$d(u_0,u_{2})=2$ and $d(v_{3s+1},u_{2})=s+2$.
$d(u_0,u_{3t+2})=4+t$ and $d(v_{3s+1},u_{3t+2})=s-t+2$ where $1\le t< s$.
When $1\le t\le\frac{s-3}{2}$, $d(u_0,u_{3t+2})<d(v_{3s+1},u_{3t+2})$.
When $\frac{s-1}{2}\le t< s$, $d(u_0,u_{3t+2})>d(v_{3s+1},u_{3t+2})$.

Note that $u_0\in W^1_{u_0v_{3s+1}}$ and $v_{3s+1}\in W^1_{v_{3s+1}u_0}$.
Combined with the above discussion, $|W^1_{u_0v_{3s+1}}|=3s+1$ and $|W^1_{v_{3s+1}u_0}|=3s+3$.

(2b) Computation of $|W^1_{u_0v_{3s+1}}|$ and $|W^1_{v_{3s+1}u_0}|$ when $s$ is even and $s\ge 2$.

When $s=2$.

$d(u_0,v_0)=1$ and $d(v_7,v_0)=5$. $d(u_0,v_3)=2$ and $d(v_7,v_3)=4$. $d(u_0,v_6)=3$ and $d(v_7,v_6)=3$.
$d(u_0,v_1)=2$ and $d(v_7,v_1)=2$. $d(u_0,v_4)=3$ and $d(v_7,v_4)=1$.
$d(u_0,v_2)=3$ and $d(v_7,v_2)=5$. $d(u_0,v_5)=4$ and $d(v_7,v_5)=4$.
So $v_0,v_2,v_3\in W^1_{u_0v_{7}}$ and $v_4\in W^1_{v_{7}u_0}$.

$d(u_0,u_3)=3$ and $d(v_7,u_3)=3$. $d(u_0,u_6)=4$ and $d(v_7,u_6)=2$.
$d(u_0,u_1)=1$ and $d(v_7,u_1)=3$. $d(u_0,u_4)=4$ and $d(v_7,u_4)=2$. $d(u_0,u_7)=5$ and $d(v_7,u_7)=1$.
$d(u_0,u_2)=2$ and $d(v_7,u_2)=4$. $d(u_0,u_5)=5$ and $d(v_7,u_5)=4$.
So $u_1,u_2\in W^1_{u_0v_{7}}$ and $u_4,u_5,u_6,u_7\in W^1_{v_{7}u_0}$.

Note that $u_0\in W^1_{u_0v_{7}}$ and $v_{7}\in W^1_{v_{7}u_0}$.
Combined with the above discussion, $|W^1_{u_0v_{7}}|=6$ and $|W^1_{v_{7}u_0}|=6$.

When $s\ge 4$.

$d(u_0,v_{3t})=1+t$ and $d(v_{3s+1},v_{3t})=s-t+3$ where $0\le t\le s$.
When $0\le t<\frac{s+2}{2}$, $d(u_0,v_{3t})<d(v_{3s+1},v_{3t})$.
When $\frac{s+2}{2}< t\le s$, $d(u_0,v_{3t})>d(v_{3s+1},v_{3t})$.
$d(u_0,v_{3t+1})=2+t$ and $d(v_{3s+1},v_{3t+1})=s-t$ where $0\le t<s$.
When $0\le t<\frac{s-2}{2}$, $d(u_0,v_{3t+1})<d(v_{3s+1},v_{3t+1})$.
When $\frac{s-2}{2}< t<s$, $d(u_0,v_{3t+1})>d(v_{3s+1},v_{3t+1})$.
$d(u_0,v_{3t+2})=3+t$ and $d(v_{3s+1},v_{3t+2})=s-t+3$ where $0\le t<s$.
When $0\le t<\frac{s}{2}$, $d(u_0,v_{3t+2})<d(v_{3s+1},v_{3t+2})$.
When $\frac{s}{2}< t<s$, $d(u_0,v_{3t+2})>d(v_{3s+1},v_{3t+2})$.

$d(u_0,u_{3t})=2+t$ and $d(v_{3s+1},u_{3t})=s-t+2$ where $1\le t\le s$.
When $1\le t<\frac{s}{2}$, $d(u_0,u_{3t})<d(v_{3s+1},u_{3t})$.
When $\frac{s}{2}< t\le s$, $d(u_0,u_{3t})>d(v_{3s+1},u_{3t})$.
$d(u_0,u_{1})=1$ and $d(v_{3s+1},u_{1})=s+1$.
$d(u_0,u_{3t+1})=3+t$ and $d(v_{3s+1},u_{3t+1})=s-t+1$ where $1\le t\le s$.
When $1\le t<\frac{s-2}{2}$, $d(u_0,u_{3t+1})<d(v_{3s+1},u_{3t+1})$.
When $\frac{s-2}{2}< t\le s$, $d(u_0,u_{3t+1})>d(v_{3s+1},u_{3t+1})$.
$d(u_0,u_{2})=2$ and $d(v_{3s+1},u_{2})=s+2$.
$d(u_0,u_{3t+2})=4+t$ and $d(v_{3s+1},u_{3t+2})=s-t+2$ where $1\le t< s$.
When $1\le t<\frac{s-2}{2}$, $d(u_0,u_{3t+2})<d(v_{3s+1},u_{3t+2})$.
When $\frac{s-2}{2}< t< s$, $d(u_0,u_{3t+2})>d(v_{3s+1},u_{3t+2})$.

Note that $u_0\in W^1_{u_0v_{3s+1}}$ and $v_{3s+1}\in W^1_{v_{3s+1}u_0}$.
Combined with the above discussion, $|W^1_{u_0v_{3s+1}}|=3s-2$ and $|W^1_{v_{3s+1}u_0}|=3s$.

(3a) Computation of $|W^1_{u_0v_{3s+2}}|$ and $|W^1_{v_{3s+2}u_0}|$ when $s$ is odd and $s\ge 3$.

When $s=3$.

$d(u_0,v_0)=1$ and $d(v_{11},v_0)=7$. $d(u_0,v_3)=2$ and $d(v_{11},v_3)=6$. $d(u_0,v_6)=3$ and $d(v_{11},v_6)=5$.
$d(u_0,v_9)=d(v_{11},v_9)=4$.
$d(u_0,v_1)=2$ and $d(v_{11},v_1)=6$. $d(u_0,v_4)=3$ and $d(v_{11},v_4)=5$. $d(u_0,v_7)=4$ and $d(v_{11},v_7)=4$.
$d(u_0,v_{10})=5$ and $d(v_{11},v_{10})=3$.
$d(u_0,v_2)=3$ and $d(v_{11},v_2)=3$. $d(u_0,v_5)=4$ and $d(v_{11},v_5)=2$. $d(u_0,v_8)=5$ and $d(v_{11},v_8)=1$.
So $v_0,v_1,v_3,v_4,v_6\in W^1_{u_0v_{11}}$ and $v_5,v_8,v_{10}\in W^1_{v_{11}u_0}$.

$d(u_0,u_3)=3$ and $d(v_{11},u_3)=5$. $d(u_0,u_6)=4$ and $d(v_{11},u_6)=4$. $d(u_0,u_9)=5$ and $d(v_{11},u_9)=3$.
$d(u_0,u_1)=1$ and $d(v_{11},u_1)=5$. $d(u_0,u_4)=4$ and $d(v_{11},u_4)=4$. $d(u_0,u_7)=5$ and $d(v_{11},u_7)=3$.
$d(u_0,u_{10})=6$ and $d(v_{11},u_{10})=2$.
$d(u_0,u_2)=2$ and $d(v_{11},u_2)=4$. $d(u_0,u_5)=5$ and $d(v_{11},u_5)=3$. $d(u_0,u_8)=6$ and $d(v_{11},u_8)=2$.
$d(u_0,u_{11})=7$ and $d(v_{11},u_{11})=1$.
So $u_1,u_2,u_3\in W^1_{u_0v_{11}}$ and $u_5,u_7,u_8,u_9,u_{10},u_{11}\in W^1_{v_{11}u_0}$.

Note that $u_0\in W^1_{u_0v_{11}}$ and $v_{11}\in W^1_{v_{11}u_0}$.
Combined with the above discussion, $|W^1_{u_0v_{11}}|=9$ and $|W^1_{v_{11}u_0}|=10$.

When $s\ge 5$.

$d(u_0,v_{3t})=1+t$ and $d(v_{3s+2},v_{3t})=s-t+4$ where $0\le t\le s$.
When $0\le t<\frac{s+3}{2}$, $d(u_0,v_{3t})<d(v_{3s+2},v_{3t})$.
When $\frac{s+3}{2}<t\le s$, $d(u_0,v_{3t})>d(v_{3s+2},v_{3t})$.
$d(u_0,v_{3t+1})=2+t$ and $d(v_{3s+2},v_{3t+1})=s-t+3$ where $0\le t\le s$.
When $0\le t<\frac{s+1}{2}$, $d(u_0,v_{3t+1})<d(v_{3s+2},v_{3t+1})$.
When $\frac{s+1}{2}<t\le s$, $d(u_0,v_{3t+1})>d(v_{3s+2},v_{3t+1})$.
$d(u_0,v_{3t+2})=3+t$ and $d(v_{3s+2},v_{3t+2})=s-t$ where $0\le t<s$.
When $0\le t<\frac{s-3}{2}$, $d(u_0,v_{3t+2})<d(v_{3s+2},v_{3t+2})$.
When $\frac{s-3}{2}<t<s$, $d(u_0,v_{3t+2})>d(v_{3s+2},v_{3t+2})$.

$d(u_0,u_{3t})=2+t$ and $d(v_{3s+2},u_{3t})=s-t+3$ where $1\le t\le s$.
When $1\le t<\frac{s+1}{2}$, $d(u_0,u_{3t})<d(v_{3s+2},u_{3t})$.
When $\frac{s+1}{2}<t\le s$, $d(u_0,u_{3t})>d(v_{3s+2},u_{3t})$.
$d(u_0,u_{1})=1$ and $d(v_{3s+2},u_{1})=s+2$.
$d(u_0,u_{3t+1})=3+t$ and $d(v_{3s+2},u_{3t+1})=s-t+2$ where $1\le t\le s$.
When $1\le t<\frac{s-1}{2}$, $d(u_0,u_{3t+1})<d(v_{3s+2},u_{3t+1})$.
When $\frac{s-1}{2}<t\le s$, $d(u_0,u_{3t+1})>d(v_{3s+2},u_{3t+1})$.
$d(u_0,u_{2})=2$ and $d(v_{3s+2},u_{2})=s+1$.
$d(u_0,u_{3t+2})=4+t$ and $d(v_{3s+2},u_{3t+2})=s-t+1$ where $1\le t\le s$.
When $1\le t<\frac{s-3}{2}$, $d(u_0,u_{3t+2})<d(v_{3s+2},u_{3t+2})$.
When $\frac{s-3}{2}<t\le s$, $d(u_0,u_{3t+2})>d(v_{3s+2},u_{3t+2})$.

Note that $u_0\in W^1_{u_0v_{3s+2}}$ and $v_{3s+2}\in W^1_{v_{3s+2}u_0}$.
Combined with the above discussion, $|W^1_{u_0v_{3s+2}}|=3s-1$ and $|W^1_{v_{3s+2}u_0}|=3s+1$.

(3b) Computation of $|W^1_{u_0v_{3s+2}}|$ and $|W^1_{v_{3s+2}u_0}|$ when $s$ is even and $s\ge 2$.

When $s=2$.

$d(u_0,v_0)=1$ and $d(v_8,v_0)=6$. $d(u_0,v_3)=2$ and $d(v_8,v_3)=5$. $d(u_0,v_6)=3$ and $d(v_8,v_6)=4$.
$d(u_0,v_1)=2$ and $d(v_8,v_1)=5$. $d(u_0,v_4)=3$ and $d(v_8,v_4)=4$. $d(u_0,v_7)=4$ and $d(v_8,v_7)=3$.
$d(u_0,v_2)=3$ and $d(v_8,v_2)=2$. $d(u_0,v_5)=4$ and $d(v_8,v_5)=1$.
So $v_0,v_1,v_3,v_4,v_6\in W^1_{u_0v_{8}}$ and $v_2,v_5,v_7\in W^1_{v_{8}u_0}$.

$d(u_0,u_3)=3$ and $d(v_8,u_3)=4$. $d(u_0,u_6)=4$ and $d(v_8,u_6)=3$.
$d(u_0,u_1)=1$ and $d(v_8,u_1)=4$. $d(u_0,u_4)=4$ and $d(v_8,u_4)=3$. $d(u_0,u_7)=5$ and $d(v_8,u_7)=2$.
$d(u_0,u_2)=2$ and $d(v_8,u_2)=3$. $d(u_0,u_5)=5$ and $d(v_8,u_5)=2$. $d(u_0,u_8)=6$ and $d(v_8,u_8)=1$.
So $u_1,u_2,u_3\in W^1_{u_0v_{8}}$ and $u_4,u_5,u_6,u_7,u_8\in W^1_{v_{8}u_0}$.

Note that $u_0\in W^1_{u_0v_{8}}$ and $v_{8}\in W^1_{v_{8}u_0}$.
Combined with the above discussion, $|W^1_{u_0v_{8}}|=9$ and $|W^1_{v_{8}u_0}|=9$.

When $s\ge 4$.

$d(u_0,v_{3t})=1+t$ and $d(v_{3s+2},v_{3t})=s-t+4$ where $0\le t\le s$.
When $0\le t\le\frac{s+2}{2}$, $d(u_0,v_{3t})<d(v_{3s+2},v_{3t})$.
When $\frac{s+4}{2}\le t\le s$, $d(u_0,v_{3t})>d(v_{3s+2},v_{3t})$.
$d(u_0,v_{3t+1})=2+t$ and $d(v_{3s+2},v_{3t+1})=s-t+3$ where $0\le t\le s$.
When $0\le t\le\frac{s}{2}$, $d(u_0,v_{3t+1})<d(v_{3s+2},v_{3t+1})$.
When $\frac{s+2}{2}\le t\le s$, $d(u_0,v_{3t+1})>d(v_{3s+2},v_{3t+1})$.
$d(u_0,v_{3t+2})=3+t$ and $d(v_{3s+2},v_{3t+2})=s-t$ where $0\le t<s$.
When $0\le t\le\frac{s-4}{2}$, $d(u_0,v_{3t+2})<d(v_{3s+2},v_{3t+2})$.
When $\frac{s-2}{2}\le t<s$, $d(u_0,v_{3t+2})>d(v_{3s+2},v_{3t+2})$.

$d(u_0,u_{3t})=2+t$ and $d(v_{3s+2},u_{3t})=s-t+3$ where $1\le t\le s$.
When $1\le t\le\frac{s}{2}$, $d(u_0,u_{3t})<d(v_{3s+2},u_{3t})$.
When $\frac{s+2}{2}\le t\le s$, $d(u_0,u_{3t})>d(v_{3s+2},u_{3t})$.
$d(u_0,u_{1})=1$ and $d(v_{3s+2},u_{1})=s+2$.
$d(u_0,u_{3t+1})=3+t$ and $d(v_{3s+2},u_{3t+1})=s-t+2$ where $1\le t\le s$.
When $1\le t\le\frac{s-2}{2}$, $d(u_0,u_{3t+1})<d(v_{3s+2},u_{3t+1})$.
When $\frac{s}{2}\le t\le s$, $d(u_0,u_{3t+1})>d(v_{3s+2},u_{3t+1})$.
$d(u_0,u_{2})=2$ and $d(v_{3s+2},u_{2})=s+1$.
$d(u_0,u_{3t+2})=4+t$ and $d(v_{3s+2},u_{3t+2})=s-t+1$ where $1\le t\le s$.
When $1\le t\le\frac{s-4}{2}$, $d(u_0,u_{3t+2})<d(v_{3s+2},u_{3t+2})$.
When $\frac{s-2}{2}\le t\le s$, $d(u_0,u_{3t+2})>d(v_{3s+2},u_{3t+2})$.

Note that $u_0\in W^1_{u_0v_{3s+2}}$ and $v_{3s+2}\in W^1_{v_{3s+2}u_0}$.
Combined with the above discussion, $|W^1_{u_0v_{3s+2}}|=3s+2$ and $|W^1_{v_{3s+2}u_0}|=3s+4$.

When $n\ge 17$, from the above computation of $|W^1_{u_0v_j}|$ and $|W^1_{v_ju_0}|$ where $6\le j\le n-6$,
for any $3\le \ell<D$,
we know that there exists $j$ where $d(u_0,v_j)=\ell$ and $6\le j\le n/2$ such that $|W_{u_0v_j}|<|W_{v_ju_0}|$.
The proof of the theorem completes.

\end{proof}

\section{The proof of Theorem~\ref{T:GP(n,4)-onlyD-DB}}\label{S:GP(n,4)}

From \cite{Miklavic:2018}, the diameter of $GP(24,4)$ is $6$ and $GP(24,4)$ is $\ell$-distance-balanced
if and only if $\ell\in\{1,6\}$.
So we can suppose $n> 24$ when we prove Theorem \ref{T:GP(n,4)-onlyD-DB}. We will prove Theorem \ref{T:GP(n,4)-onlyD-DB}
via Proposition \ref{P:GP(n,4)-1}, \ref{P:GP(n,4)-2} and \ref{P:GP(n,4)-3}.

\begin{proposition}\label{P:GP(n,4)-1}
For any $n>24$, the generalized Petersen graph $GP(n,4)$ is not $1$-distance-balanced.
\end{proposition}
\begin{proof}
In $GP(n,4)$, $d(u_0,v_0)=1$ and we will prove that $|W_{u_0v_{0}}|<|W_{v_{0}u_0}|$. We divide the discussion into
the following eight cases according to the size of $n$.

(1) When $n=8m$ where $m\ge 4$.

By symmetry, we just need to consider vertices $u_i$ and $v_i$ where $1\le i\le \frac{n}{2}$.

$d(u_0,v_{4t})=1+t$ and $d(v_0,v_{4t})=t$ where $1\le t\le m$.
$d(u_0,v_{4t+1})=2+t$ and $d(v_0,v_{4t+1})=3+t$ where $0\le t< m$.
$d(u_0,v_{4t+2})=3+t$ and $d(v_0,v_{4t+2})=4+t$ where $0\le t< m$.
$d(u_0,v_{4t+3})=3+t$ and $d(v_0,v_{4t+3})=4+t$ where $0\le t< m$.

$d(u_0,u_{4t})=2+t$ and $d(v_0,u_{4t})=1+t$ where $1\le t\le m$.
$d(u_0,u_{1})=1$ and $d(v_0,u_{1})=2$.
$d(u_0,u_{4t+1})=3+t$ and $d(v_0,u_{4t+1})=2+t$ where $1\le t< m$.
$d(u_0,u_{2})=2$ and $d(v_0,u_{2})=3$.
$d(u_0,u_{4t+2})=4+t$ and $d(v_0,u_{4t+2})=3+t$ where $1\le t< m$.
$d(u_0,u_{3})=3$ and $d(v_0,u_{3})=3$.
$d(u_0,u_{4t+3})=4+t$ and $d(v_0,u_{4t+3})=3+t$ where $1\le t< m$.

Note that $u_0\in W_{u_0v_{0}}$ and $v_{0}\in W_{v_{0}u_0}$.
Combined with the above discussion, $|W_{u_0v_{0}}|=2(3m+2)+1=6m+5$
and $|W_{v_{0}u_0}|=2(5m-5)+3=10m-7$. Because $m\ge 4$, $|W_{u_0v_{0}}|<|W_{v_{0}u_0}|$.

(2) When $n=8m+1$ where $m\ge 3$.

By symmetry, we just need to consider vertices $u_i$ and $v_i$ where $1\le i\le \frac{n}{2}$.

$d(u_0,v_{4t})=1+t$ and $d(v_0,v_{4t})=t$ where $1\le t\le m$.
$d(u_0,v_{4t+1})=2+t$ and $d(v_0,v_{4t+1})=3+t$ where $0\le t\le m-2$.
$d(u_0,v_{4(m-1)+1})=d(v_0,v_{4(m-1)+1})=m+1$.
$d(u_0,v_{4t+2})=3+t$ and $d(v_0,v_{4t+2})=4+t$ where $0\le t< m$.
$d(u_0,v_{4t+3})=3+t$ and $d(v_0,v_{4t+3})=4+t$ where $0\le t< m$.

$d(u_0,u_{4t})=2+t$ and $d(v_0,u_{4t})=1+t$ where $1\le t\le m$.
$d(u_0,u_{1})=1$ and $d(v_0,u_{1})=2$.
$d(u_0,u_{4t+1})=3+t$ and $d(v_0,u_{4t+1})=2+t$ where $1\le t< m$.
$d(u_0,u_{2})=2$ and $d(v_0,u_{2})=3$.
$d(u_0,u_{4t+2})=4+t$ and $d(v_0,u_{4t+2})=3+t$ where $1\le t< m$.
$d(u_0,u_{3})=3$ and $d(v_0,u_{3})=3$.
$d(u_0,u_{4t+3})=4+t$ and $d(v_0,u_{4t+3})=3+t$ where $1\le t< m$.

Note that $u_0\in W_{u_0v_{0}}$ and $v_{0}\in W_{v_{0}u_0}$.
Combined with the above discussion, $|W_{u_0v_{0}}|=2(3m+1)+1=6m+3$
and $|W_{v_{0}u_0}|=2(5m-3)+1=10m-5$. Because $m\ge 3$, $|W_{u_0v_{0}}|<|W_{v_{0}u_0}|$.

(3) When $n=8m+2$ where $m\ge 3$.

By symmetry, we just need to consider vertices $u_i$ and $v_i$ where $1\le i\le \frac{n}{2}$.

$d(u_0,v_{4t})=1+t$ and $d(v_0,v_{4t})=t$ where $1\le t\le m$.
$d(u_0,v_{4t+1})=2+t$ and $d(v_0,v_{4t+1})=3+t$ where $0\le t\le m$.
$d(u_0,v_{4t+2})=3+t$ and $d(v_0,v_{4t+2})=4+t$ where $0\le t\le m-2$.
$d(u_0,v_{4(m-1)+2})=m+2$ and $d(v_0,v_{4(m-1)+2})=m+1$.
$d(u_0,v_{4t+3})=3+t$ and $d(v_0,v_{4t+3})=4+t$ where $0\le t< m$.

$d(u_0,u_{4t})=2+t$ and $d(v_0,u_{4t})=1+t$ where $1\le t\le m$.
$d(u_0,u_{1})=1$ and $d(v_0,u_{1})=2$.
$d(u_0,u_{4t+1})=3+t$ and $d(v_0,u_{4t+1})=2+t$ where $1\le t\le m$.
$d(u_0,u_{2})=2$ and $d(v_0,u_{2})=3$.
$d(u_0,u_{4t+2})=4+t$ and $d(v_0,u_{4t+2})=3+t$ where $1\le t< m$.
$d(u_0,u_{3})=3$ and $d(v_0,u_{3})=3$.
$d(u_0,u_{4t+3})=4+t$ and $d(v_0,u_{4t+3})=3+t$ where $1\le t< m$.

Note that $u_0\in W_{u_0v_{0}}$ and $v_{0}\in W_{v_{0}u_0}$.
Combined with the above discussion, $|W_{u_0v_{0}}|=2(3m+1)+2=6m+4$
and $|W_{v_{0}u_0}|=2(5m-2)+2=10m-2$. Because $m\ge 3$, $|W_{u_0v_{0}}|<|W_{v_{0}u_0}|$.

(4) When $n=8m+3$ where $m\ge 3$.

By symmetry, we just need to consider vertices $u_i$ and $v_i$ where $1\le i\le \frac{n}{2}$.

$d(u_0,v_{4t})=1+t$ and $d(v_0,v_{4t})=t$ where $1\le t\le m$.
$d(u_0,v_{4t+1})=2+t$ and $d(v_0,v_{4t+1})=3+t$ where $0\le t\le m$.
$d(u_0,v_{4t+2})=3+t$ and $d(v_0,v_{4t+2})=4+t$ where $0\le t< m$.
$d(u_0,v_{4t+3})=3+t$ and $d(v_0,v_{4t+3})=4+t$ where $0\le t\le m-2$.
$d(u_0,v_{4(m-1)+3})=m+2$ and $d(v_0,v_{4(m-1)+3})=m+1$.

$d(u_0,u_{4t})=2+t$ and $d(v_0,u_{4t})=1+t$ where $1\le t\le m$.
$d(u_0,u_{1})=1$ and $d(v_0,u_{1})=2$.
$d(u_0,u_{4t+1})=3+t$ and $d(v_0,u_{4t+1})=2+t$ where $1\le t\le m$.
$d(u_0,u_{2})=2$ and $d(v_0,u_{2})=3$.
$d(u_0,u_{4t+2})=4+t$ and $d(v_0,u_{4t+2})=3+t$ where $1\le t< m$.
$d(u_0,u_{3})=3$ and $d(v_0,u_{3})=3$.
$d(u_0,u_{4t+3})=4+t$ and $d(v_0,u_{4t+3})=3+t$ where $1\le t< m$.

Note that $u_0\in W_{u_0v_{0}}$ and $v_{0}\in W_{v_{0}u_0}$.
Combined with the above discussion, $|W_{u_0v_{0}}|=2(3m+2)+1=6m+5$
and $|W_{v_{0}u_0}|=2(5m-1)+1=10m-1$. Because $m\ge 3$, $|W_{u_0v_{0}}|<|W_{v_{0}u_0}|$.

(5) When $n=8m+4$ where $m\ge 3$.

By symmetry, we just need to consider vertices $u_i$ and $v_i$ where $1\le i\le \frac{n}{2}$.

$d(u_0,v_{4t})=1+t$ and $d(v_0,v_{4t})=t$ where $1\le t\le m$.
$d(u_0,v_{4t+1})=2+t$ and $d(v_0,v_{4t+1})=3+t$ where $0\le t\le m$.
$d(u_0,v_{4t+2})=3+t$ and $d(v_0,v_{4t+2})=4+t$ where $0\le t\le m$.
$d(u_0,v_{4t+3})=3+t$ and $d(v_0,v_{4t+3})=4+t$ where $0\le t< m$.

$d(u_0,u_{4t})=2+t$ and $d(v_0,u_{4t})=1+t$ where $1\le t\le m$.
$d(u_0,u_{1})=1$ and $d(v_0,u_{1})=2$.
$d(u_0,u_{4t+1})=3+t$ and $d(v_0,u_{4t+1})=2+t$ where $1\le t\le m$.
$d(u_0,u_{2})=2$ and $d(v_0,u_{2})=3$.
$d(u_0,u_{4t+2})=4+t$ and $d(v_0,u_{4t+2})=3+t$ where $1\le t\le m$.
$d(u_0,u_{3})=3$ and $d(v_0,u_{3})=3$.
$d(u_0,u_{4t+3})=4+t$ and $d(v_0,u_{4t+3})=3+t$ where $1\le t< m$.

Note that $u_0\in W_{u_0v_{0}}$ and $v_{0}\in W_{v_{0}u_0}$.
Combined with the above discussion, $|W_{u_0v_{0}}|=2(3m+3)+2=6m+8$
and $|W_{v_{0}u_0}|=2(5m-2)+2=10m-2$. Because $m\ge 3$, $|W_{u_0v_{0}}|<|W_{v_{0}u_0}|$.

(6) When $n=8m+5$ where $m\ge 3$.

By symmetry, we just need to consider vertices $u_i$ and $v_i$ where $1\le i\le \frac{n}{2}$.

$d(u_0,v_{4t})=1+t$ and $d(v_0,v_{4t})=t$ where $1\le t\le m$.
$d(u_0,v_{4t+1})=2+t$ and $d(v_0,v_{4t+1})=3+t$ where $0\le t\le m-1$.
$d(u_0,v_{4m+1})=m+2$ and $d(v_0,v_{4m+1})=m+1$.
$d(u_0,v_{4t+2})=3+t$ and $d(v_0,v_{4t+2})=4+t$ where $0\le t\le m$.
$d(u_0,v_{4t+3})=3+t$ and $d(v_0,v_{4t+3})=4+t$ where $0\le t< m$.

$d(u_0,u_{4t})=2+t$ and $d(v_0,u_{4t})=1+t$ where $1\le t\le m$.
$d(u_0,u_{1})=1$ and $d(v_0,u_{1})=2$.
$d(u_0,u_{4t+1})=3+t$ and $d(v_0,u_{4t+1})=2+t$ where $1\le t\le m$.
$d(u_0,u_{2})=2$ and $d(v_0,u_{2})=3$.
$d(u_0,u_{4t+2})=4+t$ and $d(v_0,u_{4t+2})=3+t$ where $1\le t\le m$.
$d(u_0,u_{3})=3$ and $d(v_0,u_{3})=3$.
$d(u_0,u_{4t+3})=4+t$ and $d(v_0,u_{4t+3})=3+t$ where $1\le t< m$.

Note that $u_0\in W_{u_0v_{0}}$ and $v_{0}\in W_{v_{0}u_0}$.
Combined with the above discussion, $|W_{u_0v_{0}}|=2(3m+3)+1=6m+7$
and $|W_{v_{0}u_0}|=2\times 5m+1=10m+1$. Because $m\ge 3$, $|W_{u_0v_{0}}|<|W_{v_{0}u_0}|$.

(7) When $n=8m+6$ where $m\ge 3$.

By symmetry, we just need to consider vertices $u_i$ and $v_i$ where $1\le i\le \frac{n}{2}$.

$d(u_0,v_{4t})=1+t$ and $d(v_0,v_{4t})=t$ where $1\le t\le m$.
$d(u_0,v_{4t+1})=2+t$ and $d(v_0,v_{4t+1})=3+t$ where $0\le t\le m$.
$d(u_0,v_{4t+2})=3+t$ and $d(v_0,v_{4t+2})=4+t$ where $0\le t\le m-2$.
$d(u_0,v_{4(m-1)+2})=d(v_0,v_{4(m-1)+2})=m+2$.
$d(u_0,v_{4m+2})=m+2$ and $d(v_0,v_{4m+2})=m+1$.
$d(u_0,v_{4t+3})=3+t$ and $d(v_0,v_{4t+3})=4+t$ where $0\le t\le m$.

$d(u_0,u_{4t})=2+t$ and $d(v_0,u_{4t})=1+t$ where $1\le t\le m$.
$d(u_0,u_{1})=1$ and $d(v_0,u_{1})=2$.
$d(u_0,u_{4t+1})=3+t$ and $d(v_0,u_{4t+1})=2+t$ where $1\le t\le m$.
$d(u_0,u_{2})=2$ and $d(v_0,u_{2})=3$.
$d(u_0,u_{4t+2})=4+t$ and $d(v_0,u_{4t+2})=3+t$ where $1\le t\le m-1$.
$d(u_0,u_{4m+2})=m+3$ and $d(v_0,u_{4m+2})=m+2$.
$d(u_0,u_{3})=3$ and $d(v_0,u_{3})=3$.
$d(u_0,u_{4t+3})=4+t$ and $d(v_0,u_{4t+3})=3+t$ where $1\le t\le m$.

Note that $u_0\in W_{u_0v_{0}}$ and $v_{0}\in W_{v_{0}u_0}$.
Combined with the above discussion, $|W_{u_0v_{0}}|=2(3m+2)+2=6m+6$
and $|W_{v_{0}u_0}|=2\times 5m+2=10m+2$. Because $m\ge 3$, $|W_{u_0v_{0}}|<|W_{v_{0}u_0}|$.

(8) When $n=8m+7$ where $m\ge 3$.

By symmetry, we just need to consider vertices $u_i$ and $v_i$ where $1\le i\le \frac{n}{2}$.

$d(u_0,v_{4t})=1+t$ and $d(v_0,v_{4t})=t$ where $1\le t\le m$.
$d(u_0,v_{4t+1})=2+t$ and $d(v_0,v_{4t+1})=3+t$ where $0\le t\le m$.
$d(u_0,v_{4t+2})=3+t$ and $d(v_0,v_{4t+2})=4+t$ where $0\le t\le m$.
$d(u_0,v_{4t+3})=3+t$ and $d(v_0,v_{4t+3})=4+t$ where $0\le t\le m-2$.
$d(u_0,v_{4(m-1)+3})=d(v_0,v_{4(m-1)+3})=m+2$.
$d(u_0,v_{4m+3})=m+2$ and $d(v_0,v_{4m+3})=m+1$.

$d(u_0,u_{4t})=2+t$ and $d(v_0,u_{4t})=1+t$ where $1\le t\le m$.
$d(u_0,u_{1})=1$ and $d(v_0,u_{1})=2$.
$d(u_0,u_{4t+1})=3+t$ and $d(v_0,u_{4t+1})=2+t$ where $1\le t\le m$.
$d(u_0,u_{2})=2$ and $d(v_0,u_{2})=3$.
$d(u_0,u_{4t+2})=4+t$ and $d(v_0,u_{4t+2})=3+t$ where $1\le t\le m$.
$d(u_0,u_{3})=3$ and $d(v_0,u_{3})=3$.
$d(u_0,u_{4t+3})=4+t$ and $d(v_0,u_{4t+3})=3+t$ where $1\le t\le m-1$.
$d(u_0,u_{4m+3})=m+3$ and $d(v_0,u_{4m+3})=m+2$.

Note that $u_0\in W_{u_0v_{0}}$ and $v_{0}\in W_{v_{0}u_0}$.
Combined with the above discussion, $|W_{u_0v_{0}}|=2(3m+3)+1=6m+7$
and $|W_{v_{0}u_0}|=2(5m+1)+1=10m+3$. Because $m\ge 3$, $|W_{u_0v_{0}}|<|W_{v_{0}u_0}|$.

\end{proof}

\begin{proposition}\label{P:GP(n,4)-2}
For any $n>24$, the generalized Petersen graph $GP(n,4)$ is not $2$-distance-balanced.
\end{proposition}
\begin{proof}

In $GP(n,4)$, $d(u_0,v_{-4})=2$ and we will prove that $|W_{u_0v_{-4}}|<|W_{v_{-4}u_0}|$.
Note that $v_{-4}=v_{n-4}$.

Firstly we consider vertices $v_{-1},v_{-2},v_{-3},u_{-1},u_{-2},u_{-3}$.

$d(u_0,v_{-1})=2$ and $d(v_{-4},v_{-1})=4$. $d(u_0,v_{-2})=3$ and $d(v_{-4},v_{-2})=4$. $d(u_0,v_{-3})=d(v_{-4},v_{-3})=3$.
$d(u_0,u_{-1})=1$ and $d(v_{-4},u_{-1})=3$. $d(u_0,u_{-2})=2$ and $d(v_{-4},u_{-2})=3$. $d(u_0,u_{-3})=3$ and $d(v_{-4},u_{-3})=2$.
So $u_{-1},u_{-2},v_{-1},v_{-2}\in W_{u_0v_{-4}}$ and $u_{-3}\in W_{v_{-4}u_0}$.

Next we consider vertices $v_i$ where $0\le i<n-4$ and $u_j$ where $1\le j\le n-4$.
We divide the discussion into
the following eight cases according to the size of $n$.

(1) When $n=8m$ where $m\ge 4$.

Note that $n-4=8m-4=4(2m-1)$.

$d(u_0,v_{4t})=d(v_{8m-4},v_{4t})=1+t$ when $0\le t\le m-1$.
$d(v_{8m-4},v_{4t})=2m-t-1$ and $d(u_0,v_{4t})>d(v_{8m-4},v_{4t})$ when $m\le t<2m-1$.
$d(u_0,v_{4t+1})=2+t$ and $d(u_0,v_{4t+1})<d(v_{8m-4},v_{4t+1})$ when $0\le t\le m-1$.
$d(u_0,v_{4t+1})=d(v_{8m-4},v_{4t+1})=2m-t+2$ when $m\le t< 2m-1$.
$d(u_0,v_{4t+2})=3+t$ and $d(u_0,v_{4t+2})<d(v_{8m-4},v_{4t+2})$ when $0\le t\le m-1$.
$d(u_0,v_{4t+2})=d(v_{8m-4},v_{4t+2})=2m-t+2$ when $m\le t< 2m-1$.
$d(u_0,v_{4t+3})=3+t$ and $d(u_0,v_{4t+3})<d(v_{8m-4},v_{4t+3})$ when $0\le t\le m-2$.
$d(u_0,v_{4t+3})=d(v_{8m-4},v_{4t+3})=2m-t+1$ when $m-1\le t< 2m-1$.

$d(u_0,u_{4t})=d(v_{8m-4},u_{4t})=2+t$ when $1\le t\le m-1$.
$d(v_{8m-4},u_{4t})=2m-t$ and $d(u_0,u_{4t})>d(v_{8m-4},u_{4t})$ when $m\le t\le 2m-1$.
$d(u_0,u_{1})=1$ and $d(v_{8m-4},u_{1})=2m+1$.
$d(u_0,u_{4t+1})=d(v_{8m-4},u_{4t+1})=3+t$ when $1\le t\le m-1$.
$d(v_{8m-4},u_{4t+1})=2m-t+1$ and $d(u_0,u_{4t+1})>d(v_{8m-4},u_{4t+1})$ when $m\le t< 2m-1$.
$d(u_0,u_{2})=2$ and $d(v_{8m-4},u_{2})=2m+1$.
$d(u_0,u_{4t+2})=d(v_{8m-4},u_{4t+2})=4+t$ when $1\le t\le m-2$.
$d(v_{8m-4},u_{4t+2})=2m-t+1$ and $d(u_0,u_{4t+2})>d(v_{8m-4},u_{4t+2})$ when $m-1\le t< 2m-1$.
$d(u_0,u_{3})=3$ and $d(v_{8m-4},u_{3})=2m$.
$d(u_0,u_{4t+3})=d(v_{8m-4},u_{4t+3})=4+t$ when $1\le t\le m-2$.
$d(v_{8m-4},u_{4t+3})=2m-t$ and $d(u_0,u_{4t+3})>d(v_{8m-4},u_{4t+3})$ when $m-1\le t< 2m-1$.

Note that $u_0\in W_{u_0v_{8m-4}}$ and $v_{8m-4}\in W_{v_{8m-4}u_0}$.
Combined with the above discussion, $|W_{u_0v_{8m-4}}|=3m+7$
and $|W_{v_{8m-4}u_0}|=5m$. Because $m\ge 4$, $|W_{u_0v_{8m-4}}|<|W_{v_{8m-4}u_0}|$.

(2) When $n=8m+1$ where $m\ge 3$.

Note that $n-4=8m-3=4(2m-1)+1$.

$d(u_0,v_{4t})=d(v_{8m-3},v_{4t})=1+t$ when $0\le t\le m$.
$d(u_0,v_{4t})=d(v_{8m-3},v_{4t})=2m-t+2$ when $m+1\le t\le 2m-1$.
$d(u_0,v_{4t+1})=2+t$ and $d(u_0,v_{4t+1})<d(v_{8m-3},v_{4t+1})$ when $0\le t\le m-2$.
$d(v_{8m-3},v_{4t+1})=2m-t-1$ and $d(u_0,v_{4t+1})>d(v_{8m-3},v_{4t+1})$ when $m-1\le t< 2m-1$.
$d(u_0,v_{4t+2})=3+t$ and $d(u_0,v_{4t+2})<d(v_{8m-3},v_{4t+2})$ when $0\le t\le m-1$.
$d(u_0,v_{4t+2})=d(v_{8m-3},v_{4t+2})=2m-t+2$ when $m\le t< 2m-1$.
$d(u_0,v_{4t+3})=3+t$ and $d(u_0,v_{4t+3})<d(v_{8m-3},v_{4t+3})$ when $0\le t\le m-1$.
$d(u_0,v_{4t+3})=d(v_{8m-3},v_{4t+3})=2m-t+2$ when $m\le t< 2m-1$.

$d(u_0,u_{4t})=d(v_{8m-3},u_{4t})=2+t$ when $1\le t\le m-1$.
$d(v_{8m-3},u_{4t})=2m-t+1$ and $d(u_0,u_{4t})>d(v_{8m-3},u_{4t})$ when $m\le t\le 2m-1$.
$d(u_0,u_{1})=1$ and $d(v_{8m-3},u_{1})=2m$.
$d(u_0,u_{4t+1})=d(v_{8m-3},u_{4t+1})=3+t$ when $1\le t\le m-2$.
$d(v_{8m-3},u_{4t+1})=2m-t$ and $d(u_0,u_{4t+1})>d(v_{8m-3},u_{4t+1})$ when $m-1\le t\le 2m-1$.
$d(u_0,u_{2})=2$ and $d(v_{8m-3},u_{2})=2m+1$.
$d(u_0,u_{4t+2})=d(v_{8m-3},u_{4t+2})=4+t$ when $1\le t\le m-2$.
$d(v_{8m-3},u_{4t+2})=2m-t+1$ and $d(u_0,u_{4t+2})>d(v_{8m-3},u_{4t+2})$ when $m-1\le t< 2m-1$.
$d(u_0,u_{3})=3$ and $d(v_{8m-3},u_{3})=2m+1$.
$d(u_0,u_{4t+3})=d(v_{8m-3},u_{4t+3})=4+t$ when $1\le t\le m-2$.
$d(v_{8m-3},u_{4t+3})=2m-t+1$ and $d(u_0,u_{4t+3})>d(v_{8m-3},u_{4t+3})$ when $m-1\le t< 2m-1$.

Note that $u_0\in W_{u_0v_{8m-3}}$ and $v_{8m-3}\in W_{v_{8m-3}u_0}$.
Combined with the above discussion, $|W_{u_0v_{8m-3}}|=3m+7$
and $|W_{v_{8m-3}u_0}|=5m+3$. Because $m\ge 3$, $|W_{u_0v_{8m-3}}|<|W_{v_{8m-3}u_0}|$.

(3) When $n=8m+2$ where $m\ge 3$.

Note that $n-4=8m-2=4(2m-1)+2$.

$d(u_0,v_{4t})=d(v_{8m-2},v_{4t})=1+t$ when $0\le t\le m+1$.
$d(u_0,v_{4t})=d(v_{8m-2},v_{4t})=2m-t+3$ when $m+2\le t\le 2m-1$.
$d(u_0,v_{4t+1})=2+t$ and $d(u_0,v_{4t+1})<d(v_{8m-2},v_{4t+1})$ when $0\le t\le m-1$.
$d(u_0,v_{4t+1})=d(v_{8m-2},v_{4t+1})=2m-t+2$ when $m\le t\le 2m-1$.
$d(u_0,v_{4t+2})=3+t$ and $d(u_0,v_{4t+2})<d(v_{8m-2},v_{4t+2})$ when $0\le t< m-2$.
$d(u_0,v_{4(m-2)+2})=d(v_{8m-2},v_{4(m-2)+2})=m+1$.
$d(v_{8m-2},v_{4t+2})=2m-t-1$ and $d(u_0,v_{4t+2})>d(v_{8m-2},v_{4t+2})$ when $m-2< t< 2m-1$.
$d(u_0,v_{4t+3})=3+t$ and $d(u_0,v_{4t+3})<d(v_{8m-2},v_{4t+3})$ when $0\le t\le m-1$.
$d(u_0,v_{4t+3})=d(v_{8m-2},v_{4t+3})=2m-t+2$ when $m\le t< 2m-1$.

$d(u_0,u_{4t})=d(v_{8m-2},u_{4t})=2+t$ when $1\le t\le m$.
$d(v_{8m-2},u_{4t})=2m-t+2$ and $d(u_0,u_{4t})>d(v_{8m-2},u_{4t})$ when $m+1\le t\le 2m-1$.
$d(u_0,u_{1})=1$ and $d(v_{8m-2},u_{1})=2m+1$.
$d(u_0,u_{4t+1})=d(v_{8m-2},u_{4t+1})=3+t$ when $1\le t\le m-1$.
$d(v_{8m-2},u_{4t+1})=2m-t+1$ and $d(u_0,u_{4t+1})>d(v_{8m-2},u_{4t+1})$ when $m\le t\le 2m-1$.
$d(u_0,u_{2})=2$ and $d(v_{8m-2},u_{2})=2m$.
$d(u_0,u_{4t+2})=d(v_{8m-2},u_{4t+2})=4+t$ when $1\le t\le m-2$.
$d(v_{8m-2},u_{4t+2})=2m-t$ and $d(u_0,u_{4t+2})>d(v_{8m-2},u_{4t+2})$ when $m-1\le t\le 2m-1$.
$d(u_0,u_{3})=3$ and $d(v_{8m-2},u_{3})=2m+1$.
$d(u_0,u_{4t+3})=d(v_{8m-2},u_{4t+3})=4+t$ when $1\le t\le m-2$.
$d(v_{8m-2},u_{4t+3})=2m-t+1$ and $d(u_0,u_{4t+3})>d(v_{8m-2},u_{4t+3})$ when $m-1\le t< 2m-1$.

Note that $u_0\in W_{u_0v_{8m-2}}$ and $v_{8m-2}\in W_{v_{8m-2}u_0}$.
Combined with the above discussion, $|W_{u_0v_{8m-2}}|=3m+6$
and $|W_{v_{8m-2}u_0}|=5m+2$. Because $m\ge 3$, $|W_{u_0v_{8m-2}}|<|W_{v_{8m-2}u_0}|$.

(4) When $n=8m+3$ where $m\ge 3$.

Note that $n-4=8m-1=4(2m-1)+3$.

$d(u_0,v_{4t})=d(v_{8m-1},v_{4t})=1+t$ when $0\le t\le m+1$.
$d(u_0,v_{4t})=d(v_{8m-1},v_{4t})=2m-t+3$ when $m+2\le t\le 2m-1$.
$d(u_0,v_{4t+1})=2+t$ and $d(u_0,v_{4t+1})<d(v_{8m-1},v_{4t+1})$ when $0\le t\le m$.
$d(u_0,v_{4t+1})=d(v_{8m-1},v_{4t+1})=2m-t+3$ when $m+1\le t\le 2m-1$.
$d(u_0,v_{4t+2})=3+t$ and $d(u_0,v_{4t+2})<d(v_{8m-1},v_{4t+2})$ when $0\le t\le m-1$.
$d(u_0,v_{4t+2})=d(v_{8m-1},v_{4t+2})=2m-t+2$ when $m\le t\le 2m-1$.
$d(u_0,v_{4t+3})=3+t$ and $d(u_0,v_{4t+3})<d(v_{8m-1},v_{4t+3})$ when $0\le t< m-2$.
$d(u_0,v_{4(m-2)+3})=d(v_{8m-1},v_{4(m-2)+3})=m+1$.
$d(v_{8m-1},v_{4t+3})=2m-t-1$ and $d(u_0,v_{4t+3})>d(v_{8m-1},v_{4t+3})$ when $m-2< t< 2m-1$.

$d(u_0,u_{4t})=d(v_{8m-1},u_{4t})=2+t$ when $1\le t\le m$.
$d(v_{8m-1},u_{4t})=2m-t+2$ and $d(u_0,u_{4t})>d(v_{8m-1},u_{4t})$ when $m+1\le t\le 2m-1$.
$d(u_0,u_{1})=1$ and $d(v_{8m-1},u_{1})=2m+2$.
$d(u_0,u_{4t+1})=d(v_{8m-1},u_{4t+1})=3+t$ when $1\le t\le m-1$.
$d(v_{8m-1},u_{4t+1})=2m-t+2$ and $d(u_0,u_{4t+1})>d(v_{8m-1},u_{4t+1})$ when $m\le t\le 2m-1$.
$d(u_0,u_{2})=2$ and $d(v_{8m-1},u_{2})=2m+1$.
$d(u_0,u_{4t+2})=d(v_{8m-1},u_{4t+2})=4+t$ when $1\le t\le m-2$.
$d(v_{8m-1},u_{4t+2})=2m-t+1$ and $d(u_0,u_{4t+2})>d(v_{8m-1},u_{4t+2})$ when $m-1\le t\le 2m-1$.
$d(u_0,u_{3})=3$ and $d(v_{8m-1},u_{3})=2m$.
$d(u_0,u_{4t+3})=d(v_{8m-1},u_{4t+3})=4+t$ when $1\le t\le m-2$.
$d(v_{8m-1},u_{4t+3})=2m-t$ and $d(u_0,u_{4t+3})>d(v_{8m-1},u_{4t+3})$ when $m-1\le t\le 2m-1$.

Note that $u_0\in W_{u_0v_{8m-1}}$ and $v_{8m-1}\in W_{v_{8m-1}u_0}$.
Combined with the above discussion, $|W_{u_0v_{8m-1}}|=3m+7$
and $|W_{v_{8m-1}u_0}|=5m+3$. Because $m\ge 3$, $|W_{u_0v_{8m-1}}|<|W_{v_{8m-1}u_0}|$.

(5) When $n=8m+4$ where $m\ge 3$.

Note that $n-4=8m=4\times 2m$.

$d(u_0,v_{4t})=d(v_{8m},v_{4t})=1+t$ when $0\le t\le m-1$.
$d(v_{8m},v_{4t})=2m-t$ and $d(u_0,v_{4t})>d(v_{8m},v_{4t})$ when $m\le t\le 2m-1$.
$d(u_0,v_{4t+1})=2+t$ and $d(u_0,v_{4t+1})<d(v_{8m},v_{4t+1})$ when $0\le t\le m$.
$d(u_0,v_{4t+1})=d(v_{8m},v_{4t+1})=2m-t+3$ when $m+1\le t\le 2m-1$.
$d(u_0,v_{4t+2})=3+t$ and $d(u_0,v_{4t+2})<d(v_{8m},v_{4t+2})$ when $0\le t\le m-1$.
$d(u_0,v_{4t+2})=d(v_{8m},v_{4t+2})=2m-t+3$ when $m\le t\le 2m-1$.
$d(u_0,v_{4t+3})=3+t$ and $d(u_0,v_{4t+3})<d(v_{8m},v_{4t+3})$ when $0\le t\le m-1$.
$d(u_0,v_{4t+3})=d(v_{8m},v_{4t+3})=2m-t+2$ when $m\le t\le 2m-1$.

$d(u_0,u_{4t})=d(v_{8m},u_{4t})=2+t$ when $1\le t\le m-1$.
$d(v_{8m},u_{4t})=2m-t+1$ and $d(u_0,u_{4t})>d(v_{8m},u_{4t})$ when $m\le t\le 2m$.
$d(u_0,u_{1})=1$ and $d(v_{8m},u_{1})=2m+2$.
$d(u_0,u_{4t+1})=d(v_{8m},u_{4t+1})=3+t$ when $1\le t\le m-1$.
$d(v_{8m},u_{4t+1})=2m-t+2$ and $d(u_0,u_{4t+1})>d(v_{8m},u_{4t+1})$ when $m\le t\le 2m-1$.
$d(u_0,u_{2})=2$ and $d(v_{8m},u_{2})=2m+2$.
$d(u_0,u_{4t+2})=d(v_{8m},u_{4t+2})=4+t$ when $1\le t\le m-1$.
$d(v_{8m},u_{4t+2})=2m-t+2$ and $d(u_0,u_{4t+2})>d(v_{8m},u_{4t+2})$ when $m\le t\le 2m-1$.
$d(u_0,u_{3})=3$ and $d(v_{8m},u_{3})=2m+1$.
$d(u_0,u_{4t+3})=d(v_{8m},u_{4t+3})=4+t$ when $1\le t\le m-2$.
$d(v_{8m},u_{4t+3})=2m-t+1$ and $d(u_0,u_{4t+3})>d(v_{8m},u_{4t+3})$ when $m-1\le t\le 2m-1$.

Note that $u_0\in W_{u_0v_{8m}}$ and $v_{8m}\in W_{v_{8m}u_0}$.
Combined with the above discussion, $|W_{u_0v_{8m}}|=3m+9$
and $|W_{v_{8m}u_0}|=5m+4$. Because $m\ge 3$, $|W_{u_0v_{8m}}|<|W_{v_{8m}u_0}|$.

(6) When $n=8m+5$ where $m\ge 3$.

Note that $n-4=8m+1=4\times 2m+1$.

$d(u_0,v_{4t})=d(v_{8m+1},v_{4t})=1+t$ when $0\le t\le m+1$.
$d(u_0,v_{4t})=d(v_{8m+1},v_{4t})=2m-t+3$ when $m+2\le t\le 2m$.
$d(u_0,v_{4t+1})=2+t$ and $d(u_0,v_{4t+1})<d(v_{8m+1},v_{4t+1})$ when $0\le t\le m-2$.
$d(u_0,v_{4(m-1)+1})=d(v_{8m+1},v_{4(m-1)+1})=m+1$.
$d(v_{8m+1},v_{4t+1})=2m-t$ and $d(u_0,v_{4t+1})>d(v_{8m+1},v_{4t+1})$ when $m\le t\le 2m-1$.
$d(u_0,v_{4t+2})=3+t$ and $d(u_0,v_{4t+2})<d(v_{8m+1},v_{4t+2})$ when $0\le t\le m-1$.
$d(u_0,v_{4t+2})=d(v_{8m+1},v_{4t+2})=2m-t+3$ when $m\le t\le 2m-1$.
$d(u_0,v_{4t+3})=3+t$ and $d(u_0,v_{4t+3})<d(v_{8m+1},v_{4t+3})$ when $0\le t\le m-1$.
$d(u_0,v_{4t+3})=d(v_{8m+1},v_{4t+3})=2m-t+3$ when $m\le t\le 2m-1$.

$d(u_0,u_{4t})=d(v_{8m+1},u_{4t})=2+t$ when $1\le t\le m$.
$d(v_{8m+1},u_{4t})=2m-t+2$ and $d(u_0,u_{4t})>d(v_{8m+1},u_{4t})$ when $m+1\le t\le 2m$.
$d(u_0,u_{1})=1$ and $d(v_{8m+1},u_{1})=2m+1$.
$d(u_0,u_{4t+1})=d(v_{8m+1},u_{4t+1})=3+t$ when $1\le t\le m-1$.
$d(v_{8m+1},u_{4t+1})=2m-t+1$ and $d(u_0,u_{4t+1})>d(v_{8m+1},u_{4t+1})$ when $m\le t\le 2m$.
$d(u_0,u_{2})=2$ and $d(v_{8m+1},u_{2})=2m+2$.
$d(u_0,u_{4t+2})=d(v_{8m+1},u_{4t+2})=4+t$ when $1\le t\le m-1$.
$d(v_{8m+1},u_{4t+2})=2m-t+2$ and $d(u_0,u_{4t+2})>d(v_{8m+1},u_{4t+2})$ when $m\le t\le 2m-1$.
$d(u_0,u_{3})=3$ and $d(v_{8m+1},u_{3})=2m+2$.
$d(u_0,u_{4t+3})=d(v_{8m+1},u_{4t+3})=4+t$ when $1\le t\le m-1$.
$d(v_{8m+1},u_{4t+3})=2m-t+2$ and $d(u_0,u_{4t+3})>d(v_{8m+1},u_{4t+3})$ when $m\le t\le 2m-1$.

Note that $u_0\in W_{u_0v_{8m+1}}$ and $v_{8m+1}\in W_{v_{8m+1}u_0}$.
Combined with the above discussion, $|W_{u_0v_{8m+1}}|=3m+7$
and $|W_{v_{8m+1}u_0}|=5m+3$. Because $m\ge 3$, $|W_{u_0v_{8m+1}}|<|W_{v_{8m+1}u_0}|$.

(7) When $n=8m+6$ where $m\ge 3$.

Note that $n-4=8m+2=4\times 2m+2$.

$d(u_0,v_{4t})=d(v_{8m+2},v_{4t})=1+t$ when $0\le t\le m+1$.
$d(u_0,v_{4t})=d(v_{8m+2},v_{4t})=2m-t+4$ when $m+2\le t\le 2m$.
$d(u_0,v_{4t+1})=2+t$ and $d(u_0,v_{4t+1})<d(v_{8m+2},v_{4t+1})$ when $0\le t\le m$.
$d(u_0,v_{4t+1})=d(v_{8m+2},v_{4t+1})=2m-t+3$ when $m+1\le t\le 2m$.
$d(u_0,v_{4t+2})=3+t$ and $d(u_0,v_{4t+2})<d(v_{8m+2},v_{4t+2})$ when $0\le t\le m-2$.
$d(v_{8m+2},v_{4t+2})=2m-t$ and $d(u_0,v_{4t+2})>d(v_{8m+2},v_{4t+2})$ when $m-1\le t\le 2m-1$.
$d(u_0,v_{4t+3})=3+t$ and $d(u_0,v_{4t+3})<d(v_{8m+2},v_{4t+3})$ when $0\le t\le m-1$.
$d(u_0,v_{4t+3})=d(v_{8m+2},v_{4t+3})=2m-t+3$ when $m\le t\le 2m-1$.

$d(u_0,u_{4t})=d(v_{8m+2},u_{4t})=2+t$ when $1\le t\le m$.
$d(v_{8m+2},u_{4t})=2m-t+3$ and $d(u_0,u_{4t})>d(v_{8m+2},u_{4t})$ when $m+1\le t\le 2m$.
$d(u_0,u_{1})=1$ and $d(v_{8m+2},u_{1})=2m+2$.
$d(u_0,u_{4t+1})=d(v_{8m+2},u_{4t+1})=3+t$ when $1\le t\le m-1$.
$d(v_{8m+2},u_{4t+1})=2m-t+2$ and $d(u_0,u_{4t+1})>d(v_{8m+2},u_{4t+1})$ when $m\le t\le 2m$.
$d(u_0,u_{2})=2$ and $d(v_{8m+2},u_{2})=2m+1$.
$d(u_0,u_{4t+2})=d(v_{8m+2},u_{4t+2})=4+t$ when $1\le t\le m-2$.
$d(v_{8m+2},u_{4t+2})=2m-t+1$ and $d(u_0,u_{4t+2})>d(v_{8m+2},u_{4t+2})$ when $m-1\le t\le 2m$.
$d(u_0,u_{3})=3$ and $d(v_{8m+2},u_{3})=2m+2$.
$d(u_0,u_{4t+3})=d(v_{8m+2},u_{4t+3})=4+t$ when $1\le t\le m-1$.
$d(v_{8m+2},u_{4t+3})=2m-t+2$ and $d(u_0,u_{4t+3})>d(v_{8m+2},u_{4t+3})$ when $m\le t\le 2m-1$.

Note that $u_0\in W_{u_0v_{8m+2}}$ and $v_{8m+2}\in W_{v_{8m+2}u_0}$.
Combined with the above discussion, $|W_{u_0v_{8m+2}}|=3m+8$
and $|W_{v_{8m+2}u_0}|=5m+6$. Because $m\ge 3$, $|W_{u_0v_{8m+2}}|<|W_{v_{8m+2}u_0}|$.

(8) When $n=8m+7$ where $m\ge 3$.

Note that $n-4=8m+3=4\times 2m+3$.

$d(u_0,v_{4t})=d(v_{8m+3},v_{4t})=1+t$ when $0\le t\le m+1$.
$d(u_0,v_{4t})=d(v_{8m+3},v_{4t})=2m-t+4$ when $m+2\le t\le 2m$.
$d(u_0,v_{4t+1})=2+t$ and $d(u_0,v_{4t+1})<d(v_{8m+3},v_{4t+1})$ when $0\le t\le m$.
$d(u_0,v_{4t+1})=d(v_{8m+3},v_{4t+1})=2m-t+4$ when $m+1\le t\le 2m$.
$d(u_0,v_{4t+2})=3+t$ and $d(u_0,v_{4t+2})<d(v_{8m+3},v_{4t+2})$ when $0\le t\le m-1$.
$d(u_0,v_{4t+2})=d(v_{8m+3},v_{4t+2})=2m-t+3$ when $m\le t\le 2m$.
$d(u_0,v_{4t+3})=3+t$ and $d(u_0,v_{4t+3})<d(v_{8m+3},v_{4t+3})$ when $0\le t\le m-2$.
$d(v_{8m+3},v_{4t+3})=2m-t$ and $d(u_0,v_{4t+3})>d(v_{8m+3},v_{4t+3})$ when $m-1\le t\le 2m-1$.

$d(u_0,u_{4t})=d(v_{8m+3},u_{4t})=2+t$ when $1\le t\le m$.
$d(v_{8m+3},u_{4t})=2m-t+3$ and $d(u_0,u_{4t})>d(v_{8m+3},u_{4t})$ when $m+1\le t\le 2m$.
$d(u_0,u_{1})=1$ and $d(v_{8m+3},u_{1})=2m+3$.
$d(u_0,u_{4t+1})=d(v_{8m+3},u_{4t+1})=3+t$ when $1\le t\le m$.
$d(v_{8m+3},u_{4t+1})=2m-t+3$ and $d(u_0,u_{4t+1})>d(v_{8m+3},u_{4t+1})$ when $m+1\le t\le 2m$.
$d(u_0,u_{2})=2$ and $d(v_{8m+3},u_{2})=2m+2$.
$d(u_0,u_{4t+2})=d(v_{8m+3},u_{4t+2})=4+t$ when $1\le t\le m-1$.
$d(v_{8m+3},u_{4t+2})=2m-t+2$ and $d(u_0,u_{4t+2})>d(v_{8m+3},u_{4t+2})$ when $m\le t\le 2m$.
$d(u_0,u_{3})=3$ and $d(v_{8m+3},u_{3})=2m+1$.
$d(u_0,u_{4t+3})=d(v_{8m+3},u_{4t+3})=4+t$ when $1\le t\le m-2$.
$d(v_{8m+3},u_{4t+3})=2m-t+1$ and $d(u_0,u_{4t+3})>d(v_{8m+3},u_{4t+3})$ when $m-1\le t\le 2m$.

Note that $u_0\in W_{u_0v_{8m+3}}$ and $v_{8m+3}\in W_{v_{8m+3}u_0}$.
Combined with the above discussion, $|W_{u_0v_{8m+3}}|=3m+8$
and $|W_{v_{8m+3}u_0}|=5m+6$. Because $m\ge 3$, $|W_{u_0v_{8m+3}}|<|W_{v_{8m+3}u_0}|$.

\end{proof}

\begin{proposition}\label{P:GP(n,4)-3}
For any $n>24$, the generalized Petersen graph $GP(n,4)$ is not $\ell$-distance-balanced for any $3\le \ell<\diam(GP(n,4))$.
\end{proposition}
\begin{proof}

For any $3\le\ell<D$, we first show that there exists $v_j$ such that $d(u_0,v_j)=\ell$ where $8\le j\le n/2$.
From \cite{MaG:2023}, there exists $j^*$ such that $d(u_0,u_{j^*})=D$.

When $n=8m$ $(m\ge 4)$ or $n=8m+1$ $(m\ge 3)$, from \cite{MaG:2023} we know that $j^*=4(m-1)+2$ and $D=d(u_0,u_{j^*})=m+3$.
Note that $d(u_0,v_{4s+2})=s+3$ where $2\le s\le m-1$ and $d(u_0,v_{4s})=s+1$ where $2\le s\le m$.

When $n=8m+2$ $(m\ge 3)$ or $n=8m+3$ $(m\ge 3)$, from \cite{MaG:2023} we know that $j^*=4m+1$ and $D=d(u_0,u_{j^*})=m+3$.
Note that $d(u_0,v_{4s+1})=s+2$ where $3\le s\le m$ and $d(u_0,v_{4s})=s+1$ where $2\le s\le m$.

When $n=8m+4$ $(m\ge 3)$ or $n=8m+5$ $(m\ge 3)$, from \cite{MaG:2023} we know that $j^*=4m+2$ and $D=d(u_0,u_{j^*})=m+4$.
Note that $d(u_0,v_{4s+2})=s+3$ where $2\le s\le m$ and $d(u_0,v_{4s})=s+1$ where $2\le s\le m$.

When $n=8m+6$ $(m\ge 3)$, from \cite{MaG:2023} we know that $j^*=4m+3$ and $D=d(u_0,u_{j^*})=m+4$.
Note that $d(u_0,v_{4s+3})=s+3$ where $2\le s\le m$ and $d(u_0,v_{4s})=s+1$ where $2\le s\le m$.

When $n=8m+7$ $(m\ge 3)$, from \cite{MaG:2023} we know that $j^*=4m+2$ and $D=d(u_0,u_{j^*})=m+4$.
Note that $d(u_0,v_{4s+2})=s+3$ where $2\le s\le m$ and $d(u_0,v_{4s})=s+1$ where $2\le s\le m$.

From the above discussion, there exists $j$ where $8\le j\le n/2$ such that $d(u_0,v_j)=\ell$ for any $3\le\ell<D$.
Let $V_1=\{u_i\mid 1\le i\le j-1\}\cup\{v_i\mid 1\le i\le j-1\}$ and $V_2=\{u_i\mid j+1\le i\le n-1\}\cup\{v_i\mid j+1\le i\le n-1\}$.
Let $W^1_{u_0v_j}$ be the set of vertices which are in $W_{u_0v_j}$ and also in $V_1\cup\{u_0,v_0,u_j,v_j\}$.
Let $W^1_{v_ju_0}$ be the set of vertices which are in $W_{v_ju_0}$ and also in $V_1\cup\{u_0,v_0,u_j,v_j\}$.
Let $W^2_{u_0v_j}$ be the set of vertices which are in $W_{u_0v_j}$ and also in $V_2\cup\{u_0,v_0,u_j,v_j\}$.
Let $W^2_{v_ju_0}$ be the set of vertices which are in $W_{v_ju_0}$ and also in $V_2\cup\{u_0,v_0,u_j,v_j\}$.
Because $8\le j\le n/2$, $|W^2_{u_0v_j}|=|W^1_{u_0v_{n-j}}|$ and $|W^2_{v_ju_0}|=|W^1_{v_{n-j}u_0}|$. So
$|W_{u_0v_j}|=|W^1_{u_0v_j}|+|W^2_{u_0v_j}|-2=|W^1_{u_0v_j}|+|W^1_{u_0v_{n-j}}|-2$ and
$|W_{v_ju_0}|=|W^1_{v_ju_0}|+|W^2_{v_ju_0}|-2=|W^1_{v_ju_0}|+|W^1_{v_{n-j}u_0}|-2$.
In the following we will compute $|W^1_{u_0v_j}|$ and $|W^1_{v_ju_0}|$ where $8\le j\le n-8$.
The discussion is divided into the following eight cases.

(1a) Computation of $|W^1_{u_0v_{4s}}|$ and $|W^1_{v_{4s}u_0}|$ when $s$ is odd and $s\ge 3$.

When $s=3$,

$d(u_0,v_0)=1$ and $d(v_{12},v_0)=3$. $d(u_0,v_4)=d(v_{12},v_4)=2$. $d(u_0,v_8)=3$ and $d(v_{12},v_8)=1$.
$d(u_0,v_1)=2$ and $d(v_{12},v_1)=6$. $d(u_0,v_5)=3$ and $d(v_{12},v_5)=5$. $d(u_0,v_9)=4$ and $d(v_{12},v_9)=4$.
$d(u_0,v_2)=3$ and $d(v_{12},v_2)=6$. $d(u_0,v_6)=4$ and $d(v_{12},v_6)=5$. $d(u_0,v_{10})=5$ and $d(v_{12},v_{10})=4$.
$d(u_0,v_3)=3$ and $d(v_{12},v_3)=5$. $d(u_0,v_7)=4$ and $d(v_{12},v_7)=4$. $d(u_0,v_{11})=5$ and $d(v_{12},v_{11})=3$.
So $v_0,v_1,v_2,v_3,v_5,v_6\in W^1_{u_0v_{12}}$ and $v_8,v_{10},v_{11}\in W^1_{v_{12}u_0}$.

$d(u_0,u_4)=3$ and $d(v_{12},u_4)=3$. $d(u_0,u_8)=4$ and $d(v_{12},u_8)=2$. $d(u_0,u_{12})=5$ and $d(v_{12},u_{12})=1$.
$d(u_0,u_1)=1$ and $d(v_{12},u_1)=5$. $d(u_0,u_5)=4$ and $d(v_{12},u_5)=4$. $d(u_0,u_9)=5$ and $d(v_{12},u_9)=3$.
$d(u_0,u_2)=2$ and $d(v_{12},u_2)=5$. $d(u_0,u_6)=5$ and $d(v_{12},u_6)=4$. $d(u_0,u_{10})=6$ and $d(v_{12},u_{10})=3$.
$d(u_0,u_3)=3$ and $d(v_{12},u_3)=4$. $d(u_0,u_7)=5$ and $d(v_{12},u_7)=3$. $d(u_0,u_{11})=6$ and $d(v_{12},u_{11})=2$.
So $u_1,u_2,u_3\in W^1_{u_0v_{12}}$ and $u_6,u_7,u_8,u_9,u_{10},u_{11},u_{12}\in W^1_{v_{12}u_0}$.

Note that $u_0\in W^1_{u_0v_{12}}$ and $v_{12}\in W^1_{v_{12}u_0}$.
Combined with the above discussion, $|W^1_{u_0v_{12}}|=10$ and $|W^1_{v_{12}u_0}|=11$.

When $s\ge 5$,

$d(u_0,v_{4t})=1+t$ and $d(v_{4s},v_{4t})=s-t$ where $0\le t<s$.
When $0\le t<\frac{s-1}{2}$, $d(u_0,v_{4t})<d(v_{4s},v_{4t})$.
When $\frac{s-1}{2}<t<s$, $d(u_0,v_{4t})>d(v_{4s},v_{4t})$.
$d(u_0,v_{4t+1})=2+t$ and $d(v_{4s},v_{4t+1})=s-t+3$ where $0\le t<s$.
When $0\le t<\frac{s+1}{2}$, $d(u_0,v_{4t+1})<d(v_{4s},v_{4t+1})$.
When $\frac{s+1}{2}<t<s$, $d(u_0,v_{4t+1})>d(v_{4s},v_{4t+1})$.
$d(u_0,v_{4t+2})=3+t$ and $d(v_{4s},v_{4t+2})=s-t+3$ where $0\le t<s$.
When $0\le t\le\frac{s-1}{2}$, $d(u_0,v_{4t+2})<d(v_{4s},v_{4t+2})$.
When $\frac{s+1}{2}\le t<s$, $d(u_0,v_{4t+2})>d(v_{4s},v_{4t+2})$.
$d(u_0,v_{4t+3})=3+t$ and $d(v_{4s},v_{4t+3})=s-t+2$ where $0\le t<s$.
When $0\le t<\frac{s-1}{2}$, $d(u_0,v_{4t+3})<d(v_{4s},v_{4t+3})$.
When $\frac{s-1}{2}< t<s$, $d(u_0,v_{4t+3})>d(v_{4s},v_{4t+3})$.

$d(u_0,u_{4t})=2+t$ and $d(v_{4s},u_{4t})=s-t+1$ where $1\le t\le s$.
When $1\le t<\frac{s-1}{2}$, $d(u_0,u_{4t})<d(v_{4s},u_{4t})$.
When $\frac{s-1}{2}<t\le s$, $d(u_0,u_{4t})>d(v_{4s},u_{4t})$.
$d(u_0,u_{1})=1$ and $d(v_{4s},u_{1})=s+2$.
$d(u_0,u_{4t+1})=3+t$ and $d(v_{4s},u_{4t+1})=s-t+2$ where $1\le t< s$.
When $1\le t<\frac{s-1}{2}$, $d(u_0,u_{4t+1})<d(v_{4s},u_{4t+1})$.
When $\frac{s-1}{2}<t< s$, $d(u_0,u_{4t+1})>d(v_{4s},u_{4t+1})$.
$d(u_0,u_{2})=2$ and $d(v_{4s},u_{2})=s+2$.
$d(u_0,u_{4t+2})=4+t$ and $d(v_{4s},u_{4t+2})=s-t+2$ where $1\le t< s$.
When $1\le t\le\frac{s-3}{2}$, $d(u_0,u_{4t+2})<d(v_{4s},u_{4t+2})$.
When $\frac{s-1}{2}\le t< s$, $d(u_0,u_{4t+2})>d(v_{4s},u_{4t+2})$.
$d(u_0,u_{3})=3$ and $d(v_{4s},u_{3})=s+1$.
$d(u_0,u_{4t+3})=4+t$ and $d(v_{4s},u_{4t+3})=s-t+1$ where $1\le t< s$.
When $1\le t<\frac{s-3}{2}$, $d(u_0,u_{4t+3})<d(v_{4s},u_{4t+3})$.
When $\frac{s-3}{2}< t< s$, $d(u_0,u_{4t+3})>d(v_{4s},u_{4t+3})$.

Note that $u_0\in W^1_{u_0v_{4s}}$ and $v_{4s}\in W^1_{v_{4s}u_0}$.
Combined with the above discussion, $|W^1_{u_0v_{4s}}|=4s-3$ and $|W^1_{v_{4s}u_0}|=4s-1$.

(1b) Computation of $|W^1_{u_0v_{4s}}|$ and $|W^1_{v_{4s}u_0}|$ when $s$ is even and $s\ge 2$.

When $s=2$,

$d(u_0,v_0)=1$ and $d(v_{8},v_0)=2$. $d(u_0,v_4)=2$ and $d(v_{8},v_4)=1$.
$d(u_0,v_1)=2$ and $d(v_{8},v_1)=5$. $d(u_0,v_5)=3$ and $d(v_{8},v_5)=4$.
$d(u_0,v_2)=3$ and $d(v_{8},v_2)=5$. $d(u_0,v_6)=4$ and $d(v_{8},v_6)=4$.
$d(u_0,v_3)=3$ and $d(v_{8},v_3)=4$. $d(u_0,v_7)=4$ and $d(v_{8},v_7)=3$.
So $v_0,v_1,v_2,v_3,v_5\in W^1_{u_0v_{8}}$ and $v_4,v_{7}\in W^1_{v_{8}u_0}$.

$d(u_0,u_4)=3$ and $d(v_{8},u_4)=2$. $d(u_0,u_8)=4$ and $d(v_{8},u_8)=1$.
$d(u_0,u_1)=1$ and $d(v_{8},u_1)=4$. $d(u_0,u_5)=4$ and $d(v_{8},u_5)=3$.
$d(u_0,u_2)=2$ and $d(v_{8},u_2)=4$. $d(u_0,u_6)=5$ and $d(v_{8},u_6)=3$.
$d(u_0,u_3)=3$ and $d(v_{8},u_3)=3$. $d(u_0,u_7)=5$ and $d(v_{8},u_7)=2$.
So $u_1,u_2\in W^1_{u_0v_{8}}$ and $u_4,u_5,u_6,u_7,u_{8}\in W^1_{v_{8}u_0}$.

Note that $u_0\in W^1_{u_0v_{8}}$ and $v_{8}\in W^1_{v_{8}u_0}$.
Combined with the above discussion, $|W^1_{u_0v_{8}}|=8$ and $|W^1_{v_{8}u_0}|=8$.

When $s\ge 4$,

$d(u_0,v_{4t})=1+t$ and $d(v_{4s},v_{4t})=s-t$ where $0\le t<s$.
When $0\le t\le\frac{s-2}{2}$, $d(u_0,v_{4t})<d(v_{4s},v_{4t})$.
When $\frac{s}{2}\le t<s$, $d(u_0,v_{4t})>d(v_{4s},v_{4t})$.
$d(u_0,v_{4t+1})=2+t$ and $d(v_{4s},v_{4t+1})=s-t+3$ where $0\le t<s$.
When $0\le t\le\frac{s}{2}$, $d(u_0,v_{4t+1})<d(v_{4s},v_{4t+1})$.
When $\frac{s+2}{2}\le t<s$, $d(u_0,v_{4t+1})>d(v_{4s},v_{4t+1})$.
$d(u_0,v_{4t+2})=3+t$ and $d(v_{4s},v_{4t+2})=s-t+3$ where $0\le t<s$.
When $0\le t<\frac{s}{2}$, $d(u_0,v_{4t+2})<d(v_{4s},v_{4t+2})$.
When $\frac{s}{2}< t<s$, $d(u_0,v_{4t+2})>d(v_{4s},v_{4t+2})$.
$d(u_0,v_{4t+3})=3+t$ and $d(v_{4s},v_{4t+3})=s-t+2$ where $0\le t<s$.
When $0\le t\le\frac{s-2}{2}$, $d(u_0,v_{4t+3})<d(v_{4s},v_{4t+3})$.
When $\frac{s}{2}\le t<s$, $d(u_0,v_{4t+3})>d(v_{4s},v_{4t+3})$.

$d(u_0,u_{4t})=2+t$ and $d(v_{4s},u_{4t})=s-t+1$ where $1\le t\le s$.
When $1\le t\le\frac{s-2}{2}$, $d(u_0,u_{4t})<d(v_{4s},u_{4t})$.
When $\frac{s}{2}\le t\le s$, $d(u_0,u_{4t})>d(v_{4s},u_{4t})$.
$d(u_0,u_{1})=1$ and $d(v_{4s},u_{1})=s+2$.
$d(u_0,u_{4t+1})=3+t$ and $d(v_{4s},u_{4t+1})=s-t+2$ where $1\le t< s$.
When $1\le t\le\frac{s-2}{2}$, $d(u_0,u_{4t+1})<d(v_{4s},u_{4t+1})$.
When $\frac{s}{2}\le t< s$, $d(u_0,u_{4t+1})>d(v_{4s},u_{4t+1})$.
$d(u_0,u_{2})=2$ and $d(v_{4s},u_{2})=s+2$.
$d(u_0,u_{4t+2})=4+t$ and $d(v_{4s},u_{4t+2})=s-t+2$ where $1\le t< s$.
When $1\le t<\frac{s-2}{2}$, $d(u_0,u_{4t+2})<d(v_{4s},u_{4t+2})$.
When $\frac{s-2}{2}< t< s$, $d(u_0,u_{4t+2})>d(v_{4s},u_{4t+2})$.
$d(u_0,u_{3})=3$ and $d(v_{4s},u_{3})=s+1$.
$d(u_0,u_{4t+3})=4+t$ and $d(v_{4s},u_{4t+3})=s-t+1$ where $1\le t< s$.
When $1\le t\le\frac{s-4}{2}$, $d(u_0,u_{4t+3})<d(v_{4s},u_{4t+3})$.
When $\frac{s-2}{2}\le t< s$, $d(u_0,u_{4t+3})>d(v_{4s},u_{4t+3})$.

Note that $u_0\in W^1_{u_0v_{4s}}$ and $v_{4s}\in W^1_{v_{4s}u_0}$.
Combined with the above discussion, $|W^1_{u_0v_{4s}}|=4s-1$ and $|W^1_{v_{4s}u_0}|=4s+1$.

(2a) Computation of $|W^1_{u_0v_{4s+1}}|$ and $|W^1_{v_{4s+1}u_0}|$ when $s$ is odd and $s\ge 3$.

$d(u_0,v_{4t})=1+t$ and $d(v_{4s+1},v_{4t})=s-t+3$ where $0\le t\le s$.
When $0\le t\le\frac{s+1}{2}$, $d(u_0,v_{4t})<d(v_{4s+1},v_{4t})$.
When $\frac{s+3}{2}\le t\le s$, $d(u_0,v_{4t})>d(v_{4s+1},v_{4t})$.
$d(u_0,v_{4t+1})=2+t$ and $d(v_{4s+1},v_{4t+1})=s-t$ where $0\le t<s$.
When $0\le t\le\frac{s-3}{2}$, $d(u_0,v_{4t+1})<d(v_{4s+1},v_{4t+1})$.
When $\frac{s-1}{2}\le t<s$, $d(u_0,v_{4t+1})>d(v_{4s+1},v_{4t+1})$.
$d(u_0,v_{4t+2})=3+t$ and $d(v_{4s+1},v_{4t+2})=s-t+3$ where $0\le t<s$.
When $0\le t\le\frac{s-1}{2}$, $d(u_0,v_{4t+2})<d(v_{4s+1},v_{4t+2})$.
When $\frac{s+1}{2}\le t<s$, $d(u_0,v_{4t+2})>d(v_{4s+1},v_{4t+2})$.
$d(u_0,v_{4t+3})=3+t$ and $d(v_{4s+1},v_{4t+3})=s-t+3$ where $0\le t<s$.
When $0\le t\le\frac{s-1}{2}$, $d(u_0,v_{4t+3})<d(v_{4s+1},v_{4t+3})$.
When $\frac{s+1}{2}\le t<s$, $d(u_0,v_{4t+3})>d(v_{4s+1},v_{4t+3})$.

$d(u_0,u_{4t})=2+t$ and $d(v_{4s+1},u_{4t})=s-t+2$ where $1\le t\le s$.
When $1\le t\le\frac{s-1}{2}$, $d(u_0,u_{4t})<d(v_{4s+1},u_{4t})$.
When $\frac{s+1}{2}\le t\le s$, $d(u_0,u_{4t})>d(v_{4s+1},u_{4t})$.
$d(u_0,u_{1})=1$ and $d(v_{4s+1},u_{1})=s+1$.
$d(u_0,u_{4t+1})=3+t$ and $d(v_{4s+1},u_{4t+1})=s-t+1$ where $1\le t\le s$.
When $1\le t\le\frac{s-3}{2}$, $d(u_0,u_{4t+1})<d(v_{4s+1},u_{4t+1})$.
When $\frac{s-1}{2}\le t\le s$, $d(u_0,u_{4t+1})>d(v_{4s+1},u_{4t+1})$.
$d(u_0,u_{2})=2$ and $d(v_{4s+1},u_{2})=s+2$.
$d(u_0,u_{4t+2})=4+t$ and $d(v_{4s+1},u_{4t+2})=s-t+2$ where $1\le t< s$.
When $1\le t\le\frac{s-3}{2}$, $d(u_0,u_{4t+2})<d(v_{4s+1},u_{4t+2})$.
When $\frac{s-1}{2}\le t< s$, $d(u_0,u_{4t+2})>d(v_{4s+1},u_{4t+2})$.
$d(u_0,u_{3})=3$ and $d(v_{4s+1},u_{3})=s+2$.
$d(u_0,u_{4t+3})=4+t$ and $d(v_{4s+1},u_{4t+3})=s-t+2$ where $1\le t< s$.
When $1\le t\le\frac{s-3}{2}$, $d(u_0,u_{4t+3})<d(v_{4s+1},u_{4t+3})$.
When $\frac{s-1}{2}\le t< s$, $d(u_0,u_{4t+3})>d(v_{4s+1},u_{4t+3})$.

Note that $u_0\in W^1_{u_0v_{4s+1}}$ and $v_{4s+1}\in W^1_{v_{4s+1}u_0}$.
Combined with the above discussion, $|W^1_{u_0v_{4s+1}}|=4s+1$ and $|W^1_{v_{4s+1}u_0}|=4s+3$.

(2b) Computation of $|W^1_{u_0v_{4s+1}}|$ and $|W^1_{v_{4s+1}u_0}|$ when $s$ is even and $s\ge 4$.


When $s\ge 4$,

$d(u_0,v_{4t})=1+t$ and $d(v_{4s+1},v_{4t})=s-t+3$ where $0\le t\le s$.
When $0\le t<\frac{s+2}{2}$, $d(u_0,v_{4t})<d(v_{4s+1},v_{4t})$.
When $\frac{s+2}{2}< t\le s$, $d(u_0,v_{4t})>d(v_{4s+1},v_{4t})$.
$d(u_0,v_{4t+1})=2+t$ and $d(v_{4s+1},v_{4t+1})=s-t$ where $0\le t<s$.
When $0\le t<\frac{s-2}{2}$, $d(u_0,v_{4t+1})<d(v_{4s+1},v_{4t+1})$.
When $\frac{s-2}{2}< t<s$, $d(u_0,v_{4t+1})>d(v_{4s+1},v_{4t+1})$.
$d(u_0,v_{4t+2})=3+t$ and $d(v_{4s+1},v_{4t+2})=s-t+3$ where $0\le t<s$.
When $0\le t<\frac{s}{2}$, $d(u_0,v_{4t+2})<d(v_{4s+1},v_{4t+2})$.
When $\frac{s}{2}< t<s$, $d(u_0,v_{4t+2})>d(v_{4s+1},v_{4t+2})$.
$d(u_0,v_{4t+3})=3+t$ and $d(v_{4s+1},v_{4t+3})=s-t+3$ where $0\le t<s$.
When $0\le t<\frac{s}{2}$, $d(u_0,v_{4t+3})<d(v_{4s+1},v_{4t+3})$.
When $\frac{s}{2}< t<s$, $d(u_0,v_{4t+3})>d(v_{4s+1},v_{4t+3})$.

$d(u_0,u_{4t})=2+t$ and $d(v_{4s+1},u_{4t})=s-t+2$ where $1\le t\le s$.
When $1\le t<\frac{s}{2}$, $d(u_0,u_{4t})<d(v_{4s+1},u_{4t})$.
When $\frac{s}{2}< t\le s$, $d(u_0,u_{4t})>d(v_{4s+1},u_{4t})$.
$d(u_0,u_{1})=1$ and $d(v_{4s+1},u_{1})=s+1$.
$d(u_0,u_{4t+1})=3+t$ and $d(v_{4s+1},u_{4t+1})=s-t+1$ where $1\le t\le s$.
When $1\le t<\frac{s-2}{2}$, $d(u_0,u_{4t+1})<d(v_{4s+1},u_{4t+1})$.
When $\frac{s-2}{2}< t\le s$, $d(u_0,u_{4t+1})>d(v_{4s+1},u_{4t+1})$.
$d(u_0,u_{2})=2$ and $d(v_{4s+1},u_{2})=s+2$.
$d(u_0,u_{4t+2})=4+t$ and $d(v_{4s+1},u_{4t+2})=s-t+2$ where $1\le t< s$.
When $1\le t<\frac{s-2}{2}$, $d(u_0,u_{4t+2})<d(v_{4s+1},u_{4t+2})$.
When $\frac{s-2}{2}< t< s$, $d(u_0,u_{4t+2})>d(v_{4s+1},u_{4t+2})$.
$d(u_0,u_{3})=3$ and $d(v_{4s+1},u_{3})=s+2$.
$d(u_0,u_{4t+3})=4+t$ and $d(v_{4s+1},u_{4t+3})=s-t+2$ where $1\le t< s$.
When $1\le t<\frac{s-2}{2}$, $d(u_0,u_{4t+3})<d(v_{4s+1},u_{4t+3})$.
When $\frac{s-2}{2}< t< s$, $d(u_0,u_{4t+3})>d(v_{4s+1},u_{4t+3})$.

Note that $u_0\in W^1_{u_0v_{4s+1}}$ and $v_{4s+1}\in W^1_{v_{4s+1}u_0}$.
Combined with the above discussion, $|W^1_{u_0v_{4s+1}}|=4s-3$ and $|W^1_{v_{4s+1}u_0}|=4s-1$.

(3a) Computation of $|W^1_{u_0v_{4s+2}}|$ and $|W^1_{v_{4s+2}u_0}|$ when $s$ is odd and $s\ge 3$.

When $s=3$,

$d(u_0,v_0)=1$ and $d(v_{14},v_0)=7$. $d(u_0,v_4)=2$ and $d(v_{14},v_4)=6$. $d(u_0,v_8)=3$ and $d(v_{14},v_8)=5$.
$d(u_0,v_{12})=4$ and $d(v_{14},v_{12})=4$.
$d(u_0,v_1)=2$ and $d(v_{14},v_1)=6$. $d(u_0,v_5)=3$ and $d(v_{14},v_5)=5$. $d(u_0,v_9)=4$ and $d(v_{14},v_9)=4$.
$d(u_0,v_{13})=5$ and $d(v_{14},v_{13})=3$.
$d(u_0,v_2)=3$ and $d(v_{14},v_2)=3$. $d(u_0,v_6)=4$ and $d(v_{14},v_6)=2$. $d(u_0,v_{10})=5$ and $d(v_{14},v_{10})=1$.
$d(u_0,v_3)=3$ and $d(v_{14},v_3)=6$. $d(u_0,v_7)=4$ and $d(v_{14},v_7)=5$. $d(u_0,v_{11})=5$ and $d(v_{14},v_{11})=4$.
So $v_0,v_1,v_3,v_4,v_5,v_7,v_8\in W^1_{u_0v_{14}}$ and $v_6,v_{10},v_{11},v_{13}\in W^1_{v_{14}u_0}$.

$d(u_0,u_4)=3$ and $d(v_{14},u_4)=5$. $d(u_0,u_8)=4$ and $d(v_{14},u_8)=4$. $d(u_0,u_{12})=5$ and $d(v_{14},u_{12})=3$.
$d(u_0,u_1)=1$ and $d(v_{14},u_1)=5$. $d(u_0,u_5)=4$ and $d(v_{14},u_5)=4$. $d(u_0,u_9)=5$ and $d(v_{14},u_9)=3$.
$d(u_0,u_{13})=6$ and $d(v_{14},u_{13})=2$.
$d(u_0,u_2)=2$ and $d(v_{14},u_2)=4$. $d(u_0,u_6)=5$ and $d(v_{14},u_6)=3$. $d(u_0,u_{10})=6$ and $d(v_{14},u_{10})=2$.
$d(u_0,u_{14})=7$ and $d(v_{14},u_{14})=1$.
$d(u_0,u_3)=3$ and $d(v_{14},u_3)=5$. $d(u_0,u_7)=5$ and $d(v_{14},u_7)=4$. $d(u_0,u_{11})=6$ and $d(v_{14},u_{11})=3$.
So $u_1,u_2,u_3,u_4\in W^1_{u_0v_{14}}$ and $u_6,u_7,u_9,u_{10},u_{11},u_{12},u_{13},u_{14}\in W^1_{v_{14}u_0}$.

Note that $u_0\in W^1_{u_0v_{14}}$ and $v_{14}\in W^1_{v_{14}u_0}$.
Combined with the above discussion, $|W^1_{u_0v_{14}}|=12$ and $|W^1_{v_{14}u_0}|=13$.

When $s\ge 5$,

$d(u_0,v_{4t})=1+t$ and $d(v_{4s+2},v_{4t})=s-t+4$ where $0\le t\le s$.
When $0\le t<\frac{s+3}{2}$, $d(u_0,v_{4t})<d(v_{4s+2},v_{4t})$.
When $\frac{s+3}{2}< t\le s$, $d(u_0,v_{4t})>d(v_{4s+2},v_{4t})$.
$d(u_0,v_{4t+1})=2+t$ and $d(v_{4s+2},v_{4t+1})=s-t+3$ where $0\le t\le s$.
When $0\le t<\frac{s+1}{2}$, $d(u_0,v_{4t+1})<d(v_{4s+2},v_{4t+1})$.
When $\frac{s+1}{2}< t\le s$, $d(u_0,v_{4t+1})>d(v_{4s+2},v_{4t+1})$.
$d(u_0,v_{4t+2})=3+t$ and $d(v_{4s+2},v_{4t+2})=s-t$ where $0\le t<s$.
When $0\le t<\frac{s-3}{2}$, $d(u_0,v_{4t+2})<d(v_{4s+2},v_{4t+2})$.
When $\frac{s-3}{2}< t<s$, $d(u_0,v_{4t+2})>d(v_{4s+2},v_{4t+2})$.
$d(u_0,v_{4t+3})=3+t$ and $d(v_{4s+2},v_{4t+3})=s-t+3$ where $0\le t<s$.
When $0\le t\le\frac{s-1}{2}$, $d(u_0,v_{4t+3})<d(v_{4s+2},v_{4t+3})$.
When $\frac{s+1}{2}\le t<s$, $d(u_0,v_{4t+3})>d(v_{4s+2},v_{4t+3})$.

$d(u_0,u_{4t})=2+t$ and $d(v_{4s+2},u_{4t})=s-t+3$ where $1\le t\le s$.
When $1\le t<\frac{s+1}{2}$, $d(u_0,u_{4t})<d(v_{4s+2},u_{4t})$.
When $\frac{s+1}{2}< t\le s$, $d(u_0,u_{4t})>d(v_{4s+2},u_{4t})$.
$d(u_0,u_{1})=1$ and $d(v_{4s+2},u_{1})=s+2$.
$d(u_0,u_{4t+1})=3+t$ and $d(v_{4s+2},u_{4t+1})=s-t+2$ where $1\le t\le s$.
When $1\le t<\frac{s-1}{2}$, $d(u_0,u_{4t+1})<d(v_{4s+2},u_{4t+1})$.
When $\frac{s-1}{2}< t\le s$, $d(u_0,u_{4t+1})>d(v_{4s+2},u_{4t+1})$.
$d(u_0,u_{2})=2$ and $d(v_{4s+2},u_{2})=s+1$.
$d(u_0,u_{4t+2})=4+t$ and $d(v_{4s+2},u_{4t+2})=s-t+1$ where $1\le t\le s$.
When $1\le t<\frac{s-3}{2}$, $d(u_0,u_{4t+2})<d(v_{4s+2},u_{4t+2})$.
When $\frac{s-3}{2}< t\le s$, $d(u_0,u_{4t+2})>d(v_{4s+2},u_{4t+2})$.
$d(u_0,u_{3})=3$ and $d(v_{4s+2},u_{3})=s+2$.
$d(u_0,u_{4t+3})=4+t$ and $d(v_{4s+2},u_{4t+3})=s-t+2$ where $1\le t< s$.
When $1\le t\le\frac{s-3}{2}$, $d(u_0,u_{4t+3})<d(v_{4s+2},u_{4t+3})$.
When $\frac{s-1}{2}\le t< s$, $d(u_0,u_{4t+3})>d(v_{4s+2},u_{4t+3})$.

Note that $u_0\in W^1_{u_0v_{4s+2}}$ and $v_{4s+2}\in W^1_{v_{4s+2}u_0}$.
Combined with the above discussion, $|W^1_{u_0v_{4s+2}}|=4s-1$ and $|W^1_{v_{4s+2}u_0}|=4s+1$.

(3b) Computation of $|W^1_{u_0v_{4s+2}}|$ and $|W^1_{v_{4s+2}u_0}|$ when $s$ is even and $s\ge 2$.

When $s=2$,

$d(u_0,v_0)=1$ and $d(v_{10},v_0)=6$. $d(u_0,v_4)=2$ and $d(v_{10},v_4)=5$. $d(u_0,v_8)=3$ and $d(v_{10},v_8)=4$.
$d(u_0,v_1)=2$ and $d(v_{10},v_1)=5$. $d(u_0,v_5)=3$ and $d(v_{10},v_5)=4$. $d(u_0,v_9)=4$ and $d(v_{10},v_9)=3$.
$d(u_0,v_2)=3$ and $d(v_{10},v_2)=2$. $d(u_0,v_6)=4$ and $d(v_{10},v_6)=1$.
$d(u_0,v_3)=3$ and $d(v_{10},v_3)=5$. $d(u_0,v_7)=4$ and $d(v_{10},v_7)=4$.
So $v_0,v_1,v_3,v_4,v_5,v_8\in W^1_{u_0v_{10}}$ and $v_2,v_{6},v_9\in W^1_{v_{10}u_0}$.

$d(u_0,u_4)=3$ and $d(v_{10},u_4)=4$. $d(u_0,u_8)=4$ and $d(v_{10},u_8)=3$.
$d(u_0,u_1)=1$ and $d(v_{10},u_1)=4$. $d(u_0,u_5)=4$ and $d(v_{10},u_5)=3$. $d(u_0,u_9)=5$ and $d(v_{10},u_9)=2$.
$d(u_0,u_2)=2$ and $d(v_{10},u_2)=3$. $d(u_0,u_6)=5$ and $d(v_{10},u_6)=2$. $d(u_0,u_{10})=6$ and $d(v_{10},u_{10})=1$.
$d(u_0,u_3)=3$ and $d(v_{10},u_3)=4$. $d(u_0,u_7)=5$ and $d(v_{10},u_7)=3$.
So $u_1,u_2,u_3,u_4\in W^1_{u_0v_{10}}$ and $u_5,u_6,u_7,u_{8},u_9,u_{10}\in W^1_{v_{10}u_0}$.

Note that $u_0\in W^1_{u_0v_{10}}$ and $v_{8}\in W^1_{v_{10}u_0}$.
Combined with the above discussion, $|W^1_{u_0v_{10}}|=11$ and $|W^1_{v_{10}u_0}|=10$.

When $s\ge 4$,

$d(u_0,v_{4t})=1+t$ and $d(v_{4s+2},v_{4t})=s-t+4$ where $0\le t\le s$.
When $0\le t\le\frac{s+2}{2}$, $d(u_0,v_{4t})<d(v_{4s+2},v_{4t})$.
When $\frac{s+4}{2}\le t\le s$, $d(u_0,v_{4t})>d(v_{4s+2},v_{4t})$.
$d(u_0,v_{4t+1})=2+t$ and $d(v_{4s+2},v_{4t+1})=s-t+3$ where $0\le t\le s$.
When $0\le t\le\frac{s}{2}$, $d(u_0,v_{4t+1})<d(v_{4s+2},v_{4t+1})$.
When $\frac{s+2}{2}\le t\le s$, $d(u_0,v_{4t+1})>d(v_{4s+2},v_{4t+1})$.
$d(u_0,v_{4t+2})=3+t$ and $d(v_{4s+2},v_{4t+2})=s-t$ where $0\le t<s$.
When $0\le t\le\frac{s-4}{2}$, $d(u_0,v_{4t+2})<d(v_{4s+2},v_{4t+2})$.
When $\frac{s-2}{2}\le t<s$, $d(u_0,v_{4t+2})>d(v_{4s+2},v_{4t+2})$.
$d(u_0,v_{4t+3})=3+t$ and $d(v_{4s+2},v_{4t+3})=s-t+3$ where $0\le t<s$.
When $0\le t<\frac{s}{2}$, $d(u_0,v_{4t+3})<d(v_{4s+2},v_{4t+3})$.
When $\frac{s}{2}< t<s$, $d(u_0,v_{4t+3})>d(v_{4s+2},v_{4t+3})$.

$d(u_0,u_{4t})=2+t$ and $d(v_{4s+2},u_{4t})=s-t+3$ where $1\le t\le s$.
When $1\le t\le\frac{s}{2}$, $d(u_0,u_{4t})<d(v_{4s+2},u_{4t})$.
When $\frac{s+2}{2}\le t\le s$, $d(u_0,u_{4t})>d(v_{4s+2},u_{4t})$.
$d(u_0,u_{1})=1$ and $d(v_{4s+2},u_{1})=s+2$.
$d(u_0,u_{4t+1})=3+t$ and $d(v_{4s+2},u_{4t+1})=s-t+2$ where $1\le t\le s$.
When $1\le t\le\frac{s-2}{2}$, $d(u_0,u_{4t+1})<d(v_{4s+2},u_{4t+1})$.
When $\frac{s}{2}\le t\le s$, $d(u_0,u_{4t+1})>d(v_{4s+2},u_{4t+1})$.
$d(u_0,u_{2})=2$ and $d(v_{4s+2},u_{2})=s+1$.
$d(u_0,u_{4t+2})=4+t$ and $d(v_{4s+2},u_{4t+2})=s-t+1$ where $1\le t\le s$.
When $1\le t\le\frac{s-4}{2}$, $d(u_0,u_{4t+2})<d(v_{4s+2},u_{4t+2})$.
When $\frac{s-2}{2}\le t\le s$, $d(u_0,u_{4t+2})>d(v_{4s+2},u_{4t+2})$.
$d(u_0,u_{3})=3$ and $d(v_{4s+2},u_{3})=s+2$.
$d(u_0,u_{4t+3})=4+t$ and $d(v_{4s+2},u_{4t+3})=s-t+2$ where $1\le t< s$.
When $1\le t<\frac{s-2}{2}$, $d(u_0,u_{4t+3})<d(v_{4s+2},u_{4t+3})$.
When $\frac{s-2}{2}< t< s$, $d(u_0,u_{4t+3})>d(v_{4s+2},u_{4t+3})$.

Note that $u_0\in W^1_{u_0v_{4s+2}}$ and $v_{4s+2}\in W^1_{v_{4s+2}u_0}$.
Combined with the above discussion, $|W^1_{u_0v_{4s+2}}|=4s+1$ and $|W^1_{v_{4s+2}u_0}|=4s+3$.

(4a) Computation of $|W^1_{u_0v_{4s+3}}|$ and $|W^1_{v_{4s+3}u_0}|$ when $s$ is odd and $s\ge 3$.

When $s=3$,

$d(u_0,v_0)=1$ and $d(v_{15},v_0)=7$. $d(u_0,v_4)=2$ and $d(v_{15},v_4)=6$. $d(u_0,v_8)=3$ and $d(v_{15},v_8)=5$.
$d(u_0,v_{12})=4$ and $d(v_{15},v_{12})=4$.
$d(u_0,v_1)=2$ and $d(v_{15},v_1)=7$. $d(u_0,v_5)=3$ and $d(v_{15},v_5)=6$. $d(u_0,v_9)=4$ and $d(v_{15},v_9)=5$.
$d(u_0,v_{13})=5$ and $d(v_{15},v_{13})=4$.
$d(u_0,v_2)=3$ and $d(v_{15},v_2)=6$. $d(u_0,v_6)=4$ and $d(v_{15},v_6)=5$. $d(u_0,v_{10})=5$ and $d(v_{15},v_{10})=4$.
$d(u_0,v_{14})=6$ and $d(v_{15},v_{14})=3$.
$d(u_0,v_3)=3$ and $d(v_{15},v_3)=3$. $d(u_0,v_7)=4$ and $d(v_{15},v_7)=2$. $d(u_0,v_{11})=5$ and $d(v_{15},v_{11})=1$.
So $v_0,v_1,v_2,v_4,v_5,v_6,v_8,v_9\in W^1_{u_0v_{15}}$ and $v_7,v_{10},v_{11},v_{13},v_{14}\in W^1_{v_{15}u_0}$.

$d(u_0,u_4)=3$ and $d(v_{15},u_4)=5$. $d(u_0,u_8)=4$ and $d(v_{15},u_8)=4$. $d(u_0,u_{12})=5$ and $d(v_{15},u_{12})=3$.
$d(u_0,u_1)=1$ and $d(v_{15},u_1)=6$. $d(u_0,u_5)=4$ and $d(v_{15},u_5)=5$. $d(u_0,u_9)=5$ and $d(v_{15},u_9)=4$.
$d(u_0,u_{13})=6$ and $d(v_{15},u_{13})=3$.
$d(u_0,u_2)=2$ and $d(v_{15},u_2)=5$. $d(u_0,u_6)=5$ and $d(v_{15},u_6)=4$. $d(u_0,u_{10})=6$ and $d(v_{15},u_{10})=3$.
$d(u_0,u_{14})=7$ and $d(v_{15},u_{14})=2$.
$d(u_0,u_3)=3$ and $d(v_{15},u_3)=4$. $d(u_0,u_7)=5$ and $d(v_{15},u_7)=3$. $d(u_0,u_{11})=6$ and $d(v_{15},u_{11})=2$.
$d(u_0,u_{15})=7$ and $d(v_{15},u_{15})=1$.
So $u_1,u_2,u_3,u_4,u_5\in W^1_{u_0v_{15}}$ and $u_6,u_7,u_9,u_{10},u_{11},u_{12},u_{13},u_{14},u_{15}\in W^1_{v_{15}u_0}$.

Note that $u_0\in W^1_{u_0v_{15}}$ and $v_{15}\in W^1_{v_{15}u_0}$.
Combined with the above discussion, $|W^1_{u_0v_{15}}|=14$ and $|W^1_{v_{15}u_0}|=15$.

When $s\ge 5$,

$d(u_0,v_{4t})=1+t$ and $d(v_{4s+3},v_{4t})=s-t+4$ where $0\le t\le s$.
When $0\le t<\frac{s+3}{2}$, $d(u_0,v_{4t})<d(v_{4s+3},v_{4t})$.
When $\frac{s+3}{2}< t\le s$, $d(u_0,v_{4t})>d(v_{4s+3},v_{4t})$.
$d(u_0,v_{4t+1})=2+t$ and $d(v_{4s+3},v_{4t+1})=s-t+4$ where $0\le t\le s$.
When $0\le t\le\frac{s+1}{2}$, $d(u_0,v_{4t+1})<d(v_{4s+3},v_{4t+1})$.
When $\frac{s+3}{2}\le t\le s$, $d(u_0,v_{4t+1})>d(v_{4s+3},v_{4t+1})$.
$d(u_0,v_{4t+2})=3+t$ and $d(v_{4s+3},v_{4t+2})=s-t+3$ where $0\le t\le s$.
When $0\le t\le\frac{s-1}{2}$, $d(u_0,v_{4t+2})<d(v_{4s+3},v_{4t+2})$.
When $\frac{s+1}{2}\le t\le s$, $d(u_0,v_{4t+2})>d(v_{4s+3},v_{4t+2})$.
$d(u_0,v_{4t+3})=3+t$ and $d(v_{4s+3},v_{4t+3})=s-t$ where $0\le t<s$.
When $0\le t<\frac{s-3}{2}$, $d(u_0,v_{4t+3})<d(v_{4s+3},v_{4t+3})$.
When $\frac{s-3}{2}< t<s$, $d(u_0,v_{4t+3})>d(v_{4s+3},v_{4t+3})$.

$d(u_0,u_{4t})=2+t$ and $d(v_{4s+3},u_{4t})=s-t+3$ where $1\le t\le s$.
When $1\le t<\frac{s+1}{2}$, $d(u_0,u_{4t})<d(v_{4s+3},u_{4t})$.
When $\frac{s+1}{2}< t\le s$, $d(u_0,u_{4t})>d(v_{4s+3},u_{4t})$.
$d(u_0,u_{1})=1$ and $d(v_{4s+3},u_{1})=s+3$.
$d(u_0,u_{4t+1})=3+t$ and $d(v_{4s+3},u_{4t+1})=s-t+3$ where $1\le t\le s$.
When $1\le t\le\frac{s-1}{2}$, $d(u_0,u_{4t+1})<d(v_{4s+3},u_{4t+1})$.
When $\frac{s+1}{2}\le t\le s$, $d(u_0,u_{4t+1})>d(v_{4s+3},u_{4t+1})$.
$d(u_0,u_{2})=2$ and $d(v_{4s+3},u_{2})=s+2$.
$d(u_0,u_{4t+2})=4+t$ and $d(v_{4s+3},u_{4t+2})=s-t+2$ where $1\le t\le s$.
When $1\le t\le\frac{s-3}{2}$, $d(u_0,u_{4t+2})<d(v_{4s+3},u_{4t+2})$.
When $\frac{s-1}{2}\le t\le s$, $d(u_0,u_{4t+2})>d(v_{4s+3},u_{4t+2})$.
$d(u_0,u_{3})=3$ and $d(v_{4s+3},u_{3})=s+1$.
$d(u_0,u_{4t+3})=4+t$ and $d(v_{4s+3},u_{4t+3})=s-t+1$ where $1\le t\le s$.
When $1\le t<\frac{s-3}{2}$, $d(u_0,u_{4t+3})<d(v_{4s+3},u_{4t+3})$.
When $\frac{s-3}{2}< t\le s$, $d(u_0,u_{4t+3})>d(v_{4s+3},u_{4t+3})$.

Note that $u_0\in W^1_{u_0v_{4s+3}}$ and $v_{4s+3}\in W^1_{v_{4s+3}u_0}$.
Combined with the above discussion, $|W^1_{u_0v_{4s+3}}|=4s+1$ and $|W^1_{v_{4s+3}u_0}|=4s+3$.

(4b) Computation of $|W^1_{u_0v_{4s+3}}|$ and $|W^1_{v_{4s+3}u_0}|$ when $s$ is even.

When $s=2$.

$d(u_0,v_0)=1$ and $d(v_{11},v_0)=6$. $d(u_0,v_4)=2$ and $d(v_{11},v_4)=5$. $d(u_0,v_8)=3$ and $d(v_{11},v_8)=4$.
$d(u_0,v_1)=2$ and $d(v_{11},v_1)=6$. $d(u_0,v_5)=3$ and $d(v_{11},v_5)=5$. $d(u_0,v_9)=4$ and $d(v_{11},v_9)=4$.
$d(u_0,v_2)=3$ and $d(v_{11},v_2)=5$. $d(u_0,v_6)=4$ and $d(v_{11},v_6)=4$. $d(u_0,v_{10})=5$ and $d(v_{11},v_{10})=3$.
$d(u_0,v_3)=3$ and $d(v_{11},v_3)=2$. $d(u_0,v_7)=4$ and $d(v_{11},v_7)=1$.
So $v_0,v_1,v_2,v_4,v_5,v_8\in W^1_{u_0v_{11}}$ and $v_3,v_7,v_{10}\in W^1_{v_{11}u_0}$.

$d(u_0,u_4)=3$ and $d(v_{11},u_4)=4$. $d(u_0,u_8)=4$ and $d(v_{11},u_8)=3$.
$d(u_0,u_1)=1$ and $d(v_{11},u_1)=5$. $d(u_0,u_5)=4$ and $d(v_{11},u_5)=4$. $d(u_0,u_9)=5$ and $d(v_{11},u_9)=3$.
$d(u_0,u_2)=2$ and $d(v_{11},u_2)=4$. $d(u_0,u_6)=5$ and $d(v_{11},u_6)=3$. $d(u_0,u_{10})=6$ and $d(v_{11},u_{10})=2$.
$d(u_0,u_3)=3$ and $d(v_{11},u_3)=3$. $d(u_0,u_7)=5$ and $d(v_{11},u_7)=2$. $d(u_0,u_{11})=6$ and $d(v_{11},u_{11})=1$.
So $u_1,u_2,u_4\in W^1_{u_0v_{11}}$ and $u_6,u_7,u_8,u_9,u_{10},u_{11}\in W^1_{v_{11}u_0}$.

Note that $u_0\in W^1_{u_0v_{11}}$ and $v_{11}\in W^1_{v_{11}u_0}$.
Combined with the above discussion, $|W^1_{u_0v_{11}}|=10$ and $|W^1_{v_{11}u_0}|=10$.

When $s\ge 4$.

$d(u_0,v_{4t})=1+t$ and $d(v_{4s+3},v_{4t})=s-t+4$ where $0\le t\le s$.
When $0\le t\le\frac{s+2}{2}$, $d(u_0,v_{4t})<d(v_{4s+3},v_{4t})$.
When $\frac{s+4}{2}\le t\le s$, $d(u_0,v_{4t})>d(v_{4s+3},v_{4t})$.
$d(u_0,v_{4t+1})=2+t$ and $d(v_{4s+3},v_{4t+1})=s-t+4$ where $0\le t\le s$.
When $0\le t<\frac{s+2}{2}$, $d(u_0,v_{4t+1})<d(v_{4s+3},v_{4t+1})$.
When $\frac{s+2}{2}< t\le s$, $d(u_0,v_{4t+1})>d(v_{4s+3},v_{4t+1})$.
$d(u_0,v_{4t+2})=3+t$ and $d(v_{4s+3},v_{4t+2})=s-t+3$ where $0\le t\le s$.
When $0\le t<\frac{s}{2}$, $d(u_0,v_{4t+2})<d(v_{4s+3},v_{4t+2})$.
When $\frac{s}{2}< t\le s$, $d(u_0,v_{4t+2})>d(v_{4s+3},v_{4t+2})$.
$d(u_0,v_{4t+3})=3+t$ and $d(v_{4s+3},v_{4t+3})=s-t$ where $0\le t<s$.
When $0\le t\le\frac{s-4}{2}$, $d(u_0,v_{4t+3})<d(v_{4s+3},v_{4t+3})$.
When $\frac{s-2}{2}\le t<s$, $d(u_0,v_{4t+3})>d(v_{4s+3},v_{4t+3})$.

$d(u_0,u_{4t})=2+t$ and $d(v_{4s+3},u_{4t})=s-t+3$ where $1\le t\le s$.
When $1\le t\le\frac{s}{2}$, $d(u_0,u_{4t})<d(v_{4s+3},u_{4t})$.
When $\frac{s+2}{2}\le t\le s$, $d(u_0,u_{4t})>d(v_{4s+3},u_{4t})$.
$d(u_0,u_{1})=1$ and $d(v_{4s+3},u_{1})=s+3$.
$d(u_0,u_{4t+1})=3+t$ and $d(v_{4s+3},u_{4t+1})=s-t+3$ where $1\le t\le s$.
When $1\le t<\frac{s}{2}$, $d(u_0,u_{4t+1})<d(v_{4s+3},u_{4t+1})$.
When $\frac{s}{2}< t\le s$, $d(u_0,u_{4t+1})>d(v_{4s+3},u_{4t+1})$.
$d(u_0,u_{2})=2$ and $d(v_{4s+3},u_{2})=s+2$.
$d(u_0,u_{4t+2})=4+t$ and $d(v_{4s+3},u_{4t+2})=s-t+2$ where $1\le t\le s$.
When $1\le t<\frac{s-2}{2}$, $d(u_0,u_{4t+2})<d(v_{4s+3},u_{4t+2})$.
When $\frac{s-2}{2}< t\le s$, $d(u_0,u_{4t+2})>d(v_{4s+3},u_{4t+2})$.
$d(u_0,u_{3})=3$ and $d(v_{4s+3},u_{3})=s+1$.
$d(u_0,u_{4t+3})=4+t$ and $d(v_{4s+3},u_{4t+3})=s-t+1$ where $1\le t\le s$.
When $1\le t\le\frac{s-4}{2}$, $d(u_0,u_{4t+3})<d(v_{4s+3},u_{4t+3})$.
When $\frac{s-2}{2}\le t\le s$, $d(u_0,u_{4t+3})>d(v_{4s+3},u_{4t+3})$.

Note that $u_0\in W^1_{u_0v_{4s+3}}$ and $v_{4s+3}\in W^1_{v_{4s+3}u_0}$.
Combined with the above discussion, $|W^1_{u_0v_{4s+3}}|=4s+1$ and $|W^1_{v_{4s+3}u_0}|=4s+3$.

When $n\ge 26$, from the above computation of $|W^1_{u_0v_j}|$ and $|W^1_{v_ju_0}|$ where $8\le j\le n-8$,
for any $3\le \ell<D$,
we know that there exists $j$ where $d(u_0,v_j)=\ell$ and $8\le j\le n/2$ such that $|W_{u_0v_j}|<|W_{v_ju_0}|$.
When $n=25$, $d(u_0,v_8)=3$, $d(u_0,v_{12})=4$, $d(u_0,v_{11})=5$ and $D(GP(25,4))=6$.
From the above computation of $|W^1_{u_0v_j}|$ and $|W^1_{v_ju_0}|$,
we know that $|W_{u_0v_j}|<|W_{v_ju_0}|$ for any $j=8,11,12$.

The proof of the theorem completes.

\end{proof}

\section{Concluding remarks}\label{S:conluding}

In this paper, we prove that $GP(n,3)$ is not $\ell$-distance-balanced where $n>16$ and $1\le\ell<\diam(GP(n,3))$.
We also prove that $GP(n,4)$ is not $\ell$-distance-balanced where $n>24$ and $1\le\ell<\diam(GP(n,4))$.
The authors in \cite{Miklavic:2018} proved that $GP(n,2)$ is not $\ell$-distance-balanced where $n>11$ and $1\le\ell<\diam(GP(n,2))$.
When $k\ge 5$, Conjecture \ref{C:GP-onlyD-DB} is worth studying in the future.

Combined with the main results of this paper, Theorem \ref{T:GP-DDB} and Conjecture \ref{C:GP-onlyD-DB}, 
the following problem is worth
studying in the future.

\begin{problem}
(1) When $k\ge 5$, $k$ is odd and $\frac{k(k+1)}{2}\le n\le (k+1)^2$, whether $GP(n,k)$ is $\ell$-distance-balanced or not
where $1\le\ell<\diam(GP(n,k))$.\\
(2) When $k\ge 6$, $k$ is even and $\frac{k^2}{2}\le n\le k(k+2)$, whether $GP(n,k)$ is $\ell$-distance-balanced or not
where $1\le\ell<\diam(GP(n,k))$.
\end{problem}

\section*{Acknowledgments}

This work was supported by Shandong Provincial Natural Science Foundation of China (ZR2022MA077), the research grant NSFC (11971274) of China and IC Program of Shandong Institutions of Higher Learning For Youth Innovative Talents. Sand Klav\v{z}ar was supported by the Slovenian Research Agency (ARRS) under the grants P1-0297, J1-2452, N1-0285.

\section*{Conflict of interest statement}

On behalf of all authors, the corresponding author states that there is no conflict of interest.

\section*{Data availability statement}

Data sharing not applicable to this article as no datasets were generated or analysed during the current study.

\end{document}